\documentclass[12pt,a4paper]{amsart}

\usepackage[utf8]{inputenc}
\usepackage[T1]{fontenc}
\usepackage[american]{babel}
\usepackage{lmodern}
\usepackage{amsmath}
\usepackage{amssymb}
\usepackage{amsthm}
\usepackage{cite}
\usepackage{enumerate}
\usepackage[left=3cm,right=3cm,top=2cm,bottom=2cm]{geometry}
\usepackage{dsfont}
\usepackage{hyperref}

\newcommand{\IN}{\mathbb N}
\newcommand{\IR}{\mathbb R}

\newcommand{\B}{\mathcal B}
\newcommand{\E}{\mathcal E}
\newcommand{\F}{\mathcal F}

\renewcommand{\phi}{\varphi}

\newcommand{\1}{\mathds{1}}

\newcommand{\Capty}{\mathrm{cap}}
\newcommand{\id}{\mathrm{id}}
\newcommand{\loc}{\mathrm{loc}}
\newcommand{\lra}{\longrightarrow}
\newcommand{\supp}{\operatorname{supp}}

\providecommand{\abs}[1]{\lvert#1\rvert}
\providecommand{\norm}[1]{\lVert#1\rVert}

\theoremstyle{remark}
\newtheorem{example}{Example}[section]
\newtheorem*{remark}{Remark}

\theoremstyle{plain}
\newtheorem{proposition}[example]{Proposition}
\newtheorem{definition}[example]{Definition}
\newtheorem{theorem}[example]{Theorem}
\newtheorem{lemma}[example]{Lemma}
\newtheorem{corollary}[example]{Corollary}

\newcommand{\Hmm}[1]{\leavevmode{\marginpar{\tiny%
$\hbox to 0mm{\hspace*{-0.5mm}$\leftarrow$\hss}%
\vcenter{\vrule depth 0.1mm height 0.1mm width \the\marginparwidth}%
\hbox to 0mm{\hss$\rightarrow$\hspace*{-0.5mm}}$\\\relax\raggedright
#1}}}

\title{Geometric properties of Dirichlet forms under order isomorphisms}
\author[Lenz]{Daniel Lenz}
\address{D. Lenz, Mathematisches Institut\\Friedrich-Schiller-Universität Jena\\07737 Jena, Germany}
\email{daniel.lenz@uni-jena.de}

\author[Schmidt]{Marcel Schmidt}
\address{M. Schmidt, Mathematisches Institut\\Friedrich-Schiller-Universität Jena\\07737 Jena, Germany}
\email{schmidt.marcel@uni-jena.de}

\author[Wirth]{Melchior Wirth}
\address{M. Wirth, Mathematisches Institut\\Friedrich-Schiller-Universität Jena\\07737 Jena, Germany}
\email{melchior.wirth@uni-jena.de}
\date{\today}

\begin{document}

\begin{abstract}
We study pairs of Dirichlet forms  related by an intertwining order
isomorphisms between the associated $L^2$-spaces. We consider the
measurable, the topological and the geometric setting respectively.
In the measurable setting, we  deal with arbitrary (irreducible)
Dirichlet forms and show that any intertwining order isomorphism is
necessarily unitary (up to a constant). In the topological setting
we deal with quasi-regular forms and show that any intertwining
order isomorphism induces a quasi-homeomorphism between the
underlying spaces. In the geometric setting we deal with both
regular Dirichlet forms as well as resistance forms and essentially
show that the geometry defined by these  forms  is preserved by
intertwining  order isomorphisms. In particular, we prove in the
strongly local regular case that intertwining order isomorphisms
induce isometries with respect to the intrinsic metrics between the
underlying spaces under fairly mild assumptions. This applies to a
wide variety of metric measure spaces including
$\mathrm{RCD}(K,N)$-spaces, complete weighted Riemannian manifolds
and  complete quantum graphs. In the non-local regular case our
results cover  in particular graphs as well as fractional Laplacians
as arising in the treatment of $\alpha$-stable Lévy processes. For
resistance forms we show that intertwining order isomorphisms are
isometries with respect to the resistance metrics.

Our results can can be understood as saying that  diffusion always
determines the Hilbert space, and -- under natural compatibility
assumptions  -- the topology and the geometry respectively. As special
instances they cover earlier results for manifolds and graphs.
\end{abstract}

\maketitle

\tableofcontents

\section*{Introduction}
There is a strong interplay between geometry of a manifold, spectral
theory of its Laplace-Beltrami operator and stochastic properties of
the associated  Brownian motion. Clearly, the geometry determines
both the spectral theory and the stochastic properties and a vast
literature is devoted to this topic.

On the fundamental level also the converse is of interest. Indeed, a
famous question of Kac asks whether the spectral theory of the
Laplacian determines the geometry. This question was originally
asked for the two dimensional setting \cite{Kac66} and has triggered
a substantial  research over the time.  Starting with
 Milnor's counterexample in 16 dimensions \cite{Mil} over Sunada's
general method \cite{Sun} it took quite while till a negative answer
was given  in two dimensions  by Gordon, Webb and Wolpert in
\cite{GWW}.  In a similar spirit one may ask whether the diffusion
determines the geometry. This question was brought up by Arendt in
\cite{Are02} and a positive answer was given there for domains in Euclidean space satisfying a mild regularity assumption.  Later a
positive answer was given for manifolds  by Arendt,
Biegert, ter Elst  in \cite{ABE12},
see also \cite{AtE} for a short proof in the compact case and \cite{Are01} for
related material.

Now, Brownian motion and, more generally, symmetric  Markov
processes are not only a basic object of study on  manifolds, but
can rather be considered  on arbitrary measure spaces. A convenient
analytic framework to describe this is given by Dirichlet spaces,
i.e., measure spaces together with a Dirichlet form. More
specifically, each Dirichlet space comes naturally with a Laplace
type operator as well as a symmetric Markov process and, conversely,
any symmetric Markov process induces a Dirichlet form on the
underlying space. Prominent instances  of Dirichlet spaces are
metric measure spaces with the Cheeger energy, fractals as well as
both discrete and metric graphs. Likewise fractional Laplacians on
subsets of Euclidean space and the associated $\alpha$-stable Lévy
processes give rise to Dirichlet spaces. In all these cases the
underlying space is not only a measure space but carries further
geometric structure, and various aspects of the interplay between
this geometric structure, spectral theory of the associated
Laplacians and stochastic properties of the corresponding diffusion
process has received ample attention, see e.g. the recent articles
\cite{AGS14b,AGS15,BBCK,Che99,FG,FLS,Gig15,Kig12,KZ12,KSZ14} or the
survey collection \cite{EKKST,FL} and references therein.

Given this situation it is  very natural to ask the question whether
diffusion determines the geometry for other Dirichlet spaces than
manifolds.

In the case of  discrete graphs a positive partial answer to this
question has recently been given  in \cite{KLSW15} using the special
ingredients available in that situation.

In the present article we consider the question for arbitrary
Dirichlet spaces. As mentioned above, such  Dirichlet spaces do not
necessarily come with a specific geometry. Instead the underlying
space may only be a topological space or even just a measure space.
For this reason we deal with the question successively on the level
of measure theory and Hilbert spaces, on the level of topology and
on the level of geometry. Our main results can be summarized as
giving positive answers to the (correspondingly modified) question
on each of these levels.

We will next be  more specific and discuss the contents of the
article in further detail.

In line with the mentioned works \cite{Are02,AtE,ABE12} our point of
view is that two Markovian semigroups on (possibly) different spaces
are naturally equivalent if they are intertwined by an order
isomorphism between the corresponding $L^2$-spaces. Accordingly, our
main thrust is to find features of the underlying spaces which are
stable under existence of such intertwining order isomorphisms. In
particular, a precise version of the idea that diffusion determines
the geometry is then that equivalence of semigroups over (suitable)
metric spaces entails that there exists an isometric bijection
between the underlying spaces. Similarly, a precise version of the
idea that diffusion determines the topology is then that equivalence
of semigroups on (suitable) topological spaces entails that there
exists a homeomorphism between these spaces.

We begin our investigations in the next section with a discussion of
some background on   order isomorphism between $L^2$-spaces. By a
well-known Lamperti type result such an order isomorphism is
composed of a  measurable map with measurable a.e. inverse between the underlying spaces,
called \textit{transformation}, and a so-called \textit{scaling
function}, which is an almost everywhere strictly positive measurable
function on the range space (Proposition \ref{prop-Lamperti}). This
easily allows to compute their adjoints as well (Lemma
\ref{adjoint_U}). Given these basic results the bulk of the article
is concerned with investigating finer features of the transformation
underlying an order isomorphism between the associated $L^2$-spaces.
As mentioned already, this is done on three levels:

We first consider a measurable setting and deal with arbitrary (irreducible)
Dirichlet forms on measure spaces. We  show that any intertwining
order isomorphism  is unitary up to an overall constant (Theorem
\ref{U_unitary}). Hence, diffusion determines the Hilbert space
structure. This also gives that existence of an intertwining order
isomorphism is stronger than unitary equivalence. On the structural
level this can be seen as explanation  why diffusion -- unlike
spectral theory -- can determine the geometry. We also show that the
scaling function underlying an order isomorphism is necessarily
excessive and that it is therefore constant whenever the Dirichlet
spaces are recurrent. In the transient case however, a non-constant
scaling can not be excluded in the generality of our setting. All
this  is contained in Section \ref{sec-Orderisomorphisms}.

We then turn to a topological setting in Section
\ref{sec-Regularity}. Here, we deal with  Dirichlet forms on
topological spaces.  Of course, in order to obtain meaningful
results we have to assume some type of compatibility between the
Dirichlet form and the underlying topology. This leads us to
consider quasi-regular Dirichlet forms. These provide  the most
general setting ensuring such a compatibility.  Here, we show that
the transformation underlying an  intertwining order isomorphism
provides a  (quasi)-homeomorphism of the underlying space (Theorem
\ref{tau_quasi-homeo}). In this sense diffusion (quasi)-determines
the topology. This  provides an optimal result in the given setting.
Indeed,  in general,  one cannot expect the underlying spaces to be
homeomorphic, as can be seen from the example of a Euclidean ball
and a punctured ball with same radius and dimension, which are
indistinguishable for the Brownian motion.

This result and its proof can be seen as the heart of our article.
Loosely speaking the  proof requires the passage from measure theory
to topology. In order to achieve this, we have to overcome a major
obstacle, which was not present in any of the earlier
investigations: Specifically, we have to deal with the fact that
there is no well-defined pointwise evaluation of functions available
in this generality. In the case of manifolds this did not play a
role as one can always restrict attention to smooth functions, which
allow for pointwise evaluation. In the case of graphs this did not
play a role as the underlying space is discrete anyway. In order to
tackle this obstacle we develop some structure theory centered
around capacity and nests  for quasi-regular forms. This may be
useful in other contexts as well.

We complement the results of Section \ref{tau_quasi-homeo}   by a
study of the Beurling-Deny decomposition for quasi-regular form in
Section \ref{sec-Beurling-Deny}. In particular, we show that there
is no interaction between jump and strongly local part for
intertwined forms (Theorem \ref{transformation_decomposition}). This
insight is completely new as the mentioned earlier results only
dealt with situations in which the decomposition is trivial (i.e.
has only one term).

Let us emphasize that Section \ref{sec-Regularity} and Section
\ref{sec-Beurling-Deny} deal with quasi-regular Dirichlet forms in
full generality thereby covering also infinite dimensional cases.

Finally, we turn to geometry in the three final sections of the
article. Here, again, for meaningful results we have to assume some
type of compatibility between the geometry and the Dirichlet form.
In order to achieve this we rely on intrinsic geometry coming about
with any regular Dirichlet form on a locally compact space together
with some mild additional `smoothness' assumptions. These
assumptions are general enough to  cover all common settings of
geometric analysis. Specifically, we proceed as follows:

For strongly local regular Dirichlet forms intrinsic geometry has
been developed via the concept of the intrinsic metric in a seminal
work by Sturm \cite{Stu94}.  Our main result  (Theorem
\ref{isometry_local}) shows  that the transformation  underlying the
intertwining order isomorphism  provides an isometric homeomorphism
between the closures of the underlying spaces (where all metric
concepts are defined with respect to the intrinsic metrics). Hence,
geometry is determined by diffusion in this case. This result
considerably extends the corresponding results on manifolds  of
\cite{Are02,AtE,ABE12}. In fact, the  assumptions are satisfied for
large classes of metric measure spaces including
$\mathrm{RCD}(K,N)$-spaces, complete weighted Riemannian
manifolds and  complete quantum graphs. Details are discussed in
Section \ref{sec_strongly_local}.

For arbitrary regular Dirichlet forms a framework of  intrinsic
metrics has recently be presented by Frank, Lenz and Wingert in
\cite{FLW14}.
A new phenomenon featured in this theory is that there are in
general several non-compatible intrinsic metrics available whereas
in the strongly local case the intrinsic metric discussed in
\cite{Stu94} can be seen as the maximal intrinsic metric in the
sense of \cite{FLW14}. Our main result in this context takes care of
this multitude of intrinsic metrics by invoking the assumption of
recurrence. It states that after possible removal of closed sets
with zero capacity, the transformation underlying the order
isomorphism induces a bijection between the set of intrinsic metrics
(Theorem \ref{isometry_non-local}). So, here again,  the geometry
(as given by the family of intrinsic metrics) is determined by the
diffusion.  As a consequence we obtain a bijection between the sets
of intrinsic metrics on the whole spaces whenever points have
positive capacity. This result is new even for graphs. Theorem
\ref{isometry_non-local} also applies to  the Dirichlet forms
associated to $\alpha$-stable Levy processes thereby giving the
first non-local  examples outside of the graph setting  of diffusion
determining the geometry. All of this is discussed in Section
\ref{sec_non-local}.

Finally, there are resistance forms. Such forms play a crucial role
in the study of fractals. The class of resistance forms is not
disjoint from the class of strongly local forms or the class of
regular Dirichlet forms. However, in most prominent cases of
resistance forms intrinsic metrics as given above are not useful as
they are trivial \cite{Hin05}. 
To overcome this obstacle we use that any
 resistance
form comes with a metric, the resistance metric, which captures the
geometry. In this context our result says that the transformation
underlying the order isomorphism induces an isometric homeomorphism
with respect to the resistance metrics (Theorem
\ref{isometry_resistance}). It is the first result of this form for
fractals and covers all the usual models.  This is discussed in
Section \ref{sec-Resistance}.

As far as methods are concerned, the considerations in Section
\ref{sec_strongly_local}, Section \ref{sec_non-local} and Section
\ref{sec-Resistance}  can all be seen as building on the method
given in the  proof of Theorem \ref{tau_quasi-homeo} by additionally
using  the tools at hand in the corresponding specific situations.
In terms of results these sections  provide a rather complete
treatment of diffusion determining the geometry.  As a consequence
we not only recover the previously known results but can deal with a
wealth of new situations.

%
%

\bigskip

\textbf{Acknowledgments.} D.L. is grateful for inspiring discussions
with Matthias Keller and Peter Stollmann. Partial support of German
Research Foundation (DFG) and the German Academic Scholarship Foundation (Studienstiftung des deutschen Volkes) is gratefully acknowledged.

\section{Order isomorphisms between
\texorpdfstring{$L^p$}{Lp}-spaces}\label{sec-Basicorderisomorphism}
In this section we collect some basic results about order
isomorphisms between $L^p$-spaces. All the statements given are
essentially well-known and many hold in more general settings, but
we restrict our attention to  the situation we will need later in
the article.

\medskip

A measurable space is called \emph{standard Borel space} if it is
isomorphic to a  Polish (i.e. separable complete metric) space with
its Borel $\sigma$-algebra. We also say that a measure space is
standard Borel if the underlying measurable space is. Throughout the
section we denote by $(X_i,\B_i,m_i)$, $i\in\{1,2\}$,
$\sigma$-finite standard Borel spaces.

All equalities and inequalities in $L^p$ are to be understood as
equalities and inequalities almost everywhere (later we will have to
distinguish between almost everywhere and quasi everywhere).  As
usual we will often write a.e. for almost everywhere in the context
of functions on measure spaces.

A linear map $U\colon L^p(X_1,m_1)\lra L^p(X_2,m_2)$ is called
\emph{positivity preserving} if $f\geq 0$ implies $Uf\geq 0$. An
invertible  positivity preserving linear map with positivity
preserving inverse is called \emph{order isomorphism}. It is a
standard result that positivity preserving operators are continuous
(see e.g. \cite{AB85}, Theorem 4.3).

The structure of order isomorphisms between $L^p$-spaces is
characterized by the following Banach-Lamperti-type theorem (see
e.g. \cite{Wei84}, Proposition 5.1).
\begin{proposition}[Order isomorphism as weighted composition
operator] \label{prop-Lamperti} If $p\in[1,\infty)$ and $U\colon
L^p(X_1,m_1)\lra L^p(X_2,m_2)$ is an order isomorphism, then there
exist a measurable map \sloppy$h\colon X_2\lra (0,\infty)$ and a
measurable map $\tau\colon X_2\lra X_1$ with measurable a.e. inverse such
that
\begin{align*}
Uf=h\cdot (f\circ\tau)
\end{align*}
for all $f\in L^p(X_1,m_1)$. The maps $h$ and $\tau$ are unique up
to equality almost everywhere.
\end{proposition}

The  maps $h$ and $U$ associated with an order isomorphism according
to the previous proposition are the main players in the present
article. We call $h$   the \emph{scaling} and $\tau$ the
\emph{transformation} associated with $U$.

Given the scaling and transformation associated with an order
isomorphism, it is easy to calculate its adjoint. Here and in the
following we denote by $\phi_\#\mu$ the pushforward of the measure
$\mu$ along the map $\phi$.

\begin{lemma}[Adjoint of an order
isomorphism]\label{adjoint_U}\label{lemma-adjoint} Let
$p\in[1,\infty)$, $q$ the dual exponent, and let $U\colon
L^p(X_1,m_1)\lra L^p(X_2,m_2)$ be an order isomorphism with
associated scaling $h$ and transformation $\tau$. Then $\tau_\# m_2$
and $m_1$ are mutually absolutely continuous. Moreover, the adjoint
of $U$ is given by
\begin{align*}
U^\ast\colon L^q(X_2,m_2)\lra L^q(X_1,m_1),\,U^\ast g=\frac{d(\tau_\# m_2)}{dm_1}(hg)\circ\tau^{-1}.
\end{align*}
Furthermore, for $p=2$ one has
\begin{align*}
U^\ast U f=\frac{d(\tau_\# m_2)}{dm_1}(h\circ\tau^{-1})^2f \;\:
\mbox{ and }\;\: UU^\ast g=\frac{dm_2}{d(\tau^{-1}_\# m_1)}h^2 g
\end{align*}
for all $f\in L^2(X_1,m_1)$, $g\in L^2(X_2,m_2)$.
\end{lemma}
\begin{proof}
Let $A\in X_1$ be measurable with $m_1(A)=0$. Then $\1_A=0$ in
$L^p(X_1,m_1)$, hence $U\1_A=0$ in $L^p(X_2,m_2)$. Thus, as $h >0$
a.e. we find
\begin{align*}
\tau_\# m_2(A)=\int_{X_2} \1_A\circ\tau\,dm_2=\int_{X_2}\frac 1
h\cdot U\1_A\,dm_2=0.
\end{align*}
The converse direction works analogously by invoking $U$ instead of
$U^{-1}$.

Now we can compute the adjoint. For all $f\in L^p(X_1,m_1)$, $g\in L^q(X_2,m_2)$ we have
\begin{align*}
\langle Uf,g\rangle&=\int_{X_2}gh \cdot (f\circ\tau)\,dm_2\\
&=\int_{X_1}(gh)\circ\tau^{-1}f\,d(\tau_\# m_2)\\
&=\int_{X_1}f\frac{d(\tau_\# m_2)}{dm_1}(gh)\circ\tau^{-1}\,dm_1.
\end{align*}
Thus, $\frac{d(\tau_\# m_2)}{dm_1}(hg)\circ\tau^{-1}$ induces a continuous linear functional on $L^p(X_1,m_1)$ and must therefore be an element of $L^q$. Furthermore, the above computation shows the formula for $U^\ast$. The formulae for $UU^\ast$ and $U^\ast U$ follow easily.
\end{proof}

\section{Order isomorphisms intertwining Markovian
semigroups}\label{sec-Orderisomorphisms}
 In the previous section we have
discussed the basic structure  of order isomorphism.  In this
section we investigate the structure of order isomorphisms that
intertwine Markovian semigroups. In particular, we show that
intertwining on $L^p$ implies intertwining on $L^2$ and that
intertwining operators on $L^2$ between irreducible Markovian
semigroups are necessarily unitary up to a constant (Theorem
\ref{U_unitary}) and provide a strong rigidity result for
intertwined semigroups  in the recurrent case (Corollary
\ref{recurrent_h_const}).

Along the way we will need (and recall) various pieces of the theory
of Markovian semigroups and Dirichlet forms. For background  and
references we refer to the standard textbooks such as \cite{FOT94},
more results on (not necessarily bounded) intertwiners of Markovian
semigroups and probabilistic interpretations can be found in
\cite{PSZ17}.

\medskip

Throughout this section, $(X_1,\B_1,m_1),\,(X_2,\B_2,m_2)$ and
$(X,\B,m)$ denote $\sigma$-finite standard Borel spaces.

If $E_1,E_2$ are Banach spaces and $S_1,S_2$ are (not necessarily
bounded) operators on $E_1$ and $E_2$ respectively, an invertible
operator $U\colon E_1\lra E_2$ is said to \emph{intertwine} $S_1$
and $S_2$ if
$$UD(S_1)=D(S_2)  \;  \mbox{ and } \;  US_1 f=S_2Uf$$
for all $f\in D(S_1)$. Two families $(S^{(1)}_{\alpha})_{\alpha\in A}$
and $(S^{(2)}_{\alpha})_{\alpha\in A}$ are said to be intertwined by
$U$ if $U$ intertwines $S^{(1)}_\alpha$ and $S^{(2)}_\alpha$ for all
$\alpha\in A$.

For later use we note that two strongly continuous symmetric
contraction semigroups are intertwined by $U$ if and only if their
generators are (cf. \cite{KLSW15}, Appendix A).

Denote by $L_+(X,m)$ the space of all equivalence classes of
measurable functions $X\lra [0,\infty]$. A positivity preserving
operator $S\colon L^p(X_1,m_1)\lra L^p(X_2,m_2)$ can be uniquely
extended to a map $\tilde S\colon L_+(X_1,m_1)\lra L_+(X_2,m_2)$
that satisfies $\tilde Sf_n\nearrow \tilde Sf$ if $f_n\nearrow f$.
For some properties of this extension see \cite{Kaj17}, Proposition~1. Unless it is ambiguous, we will often omit the tilde and simply
use the same symbol for the given operator and  this extension.

If $(T_i)$ are positivity preserving operators on $L^p(X_i,m_i)$,
$i\in\{1,2\}$, which are intertwined by an order isomorphism
$U\colon L^p(X_1,m_1)\lra L^p(X_2,m_2)$, it is easily verified that
$\tilde U \tilde T_1=\tilde T_2\tilde U$.

From now on we will deal with \emph{Markovian semigroups}, that is,
strongly continuous symmetric contraction semigroups $(T_t)$ of
positivity preserving operators on $L^2$ such that $\tilde T_t 1\leq
1$ for all $t\geq 0$.

For a Markovian semigroup $(T_t)$, the restriction $(T_t|_{L^p\cap
L^2})$ extends continuously to $L^p$ for all $p\in [1,\infty)$. We
will denote this extension also by $(T_t)$. It is compatible with
the extension $\tilde{T_t}$, i.e., $T_t f = \tilde{T_t}f$ for all $f
\in L_+^p$, $t > 0$. Moreover, the semigroups on $L^p$ and $L^q$ are
mutually adjoint for dual exponents $p,q$.

\begin{lemma}
Let $(T_t^{(i)})$, $i \in \{1,2\},$ Markovian semigroups on
$L^2(X_i,m_i)$, let $p\in[1,\infty)$ and let $q$ be the dual
exponent. If $U\colon L^p(X_1,m_1)\lra L^p(X_2,m_2)$  is an order
isomorphism intertwining $(T_t^{(1)})$ and $(T_t^{(2)})$, its
adjoint $U^\ast\colon L^q(X_2,m_2)\lra L^q(X_1,m_1)$ is an order
isomorphism intertwining $(T_t^{(2)})$ and $(T_t^{(1)})$.
\end{lemma}
\begin{proof}
Denote by $\langle\cdot,\cdot\rangle$ the dual pairing of $L^p$ and
$L^q$. For $f\in L^q$ we have $f\geq 0$ if and only if  $\langle
f,g\rangle\geq 0$ for all $g\in L^p_+$.

For $f\in L_+^q(X_2,m_2)$ we have
\begin{align*}
\langle U^\ast f,g\rangle=\langle f,Ug\rangle\geq 0
\end{align*}
for all $g\in L_+^p(X_1,m_1)$. Thus, $U^\ast f\geq 0$. The same argument holds for the inverse $(U^\ast)^{-1}=(U^{-1})^\ast$ so that $U^\ast$ is an order isomorphism. The intertwining property follows by taking adjoints and using the fact that the semigroups on $L^p$ and $L^q$ are adjoint.
\end{proof}

A measurable subset $A$ of $X$ is called \emph{invariant} under the
Markovian semigroup $(T_t)$  if $\widetilde{T_t} \1_A\leq \1_A$ for
all $t\geq 0$ (for various characterizations of invariance see
\cite{Sch04}). If every invariant set is either a null set or the
complement of a null set, $(T_t)$ is called \emph{irreducible} (or
{\em ergodic}).

If $U$ is an order isomorphism intertwining Markovian semigroups
$(T_t^{(1)})$ and $(T_t^{(2)})$ and $\tau$ is the associated
transformation, it is easy to see that a measurable set $A\subset X_2$ is
invariant under $(T_t^{(2)})$ if and only if $\tau(A)$ is invariant
under $(T_t^{(1)})$. Consequently, $(T_t^{(1)})$ is irreducible if
and only if $(T_t^{(2)})$ is irreducible.

By a slight abuse of notation, we say that a measurable function $f$
is almost everywhere constant if there exists $\alpha\in\IR$ such that
$f=\alpha$ a.e.

\begin{lemma}\label{commutant_semigroup}
Let $(T_t)$ be an irreducible Markovian semigroup and let
$\phi\colon X\lra[0,\infty)$ be measurable. If $\tilde{T_t}(\phi
f)=\phi \tilde{T_t} f$ for all $f\in L_+(X,m)$ and $t\geq 0$, then
$\phi$ is almost everywhere constant.
\end{lemma}

\begin{proof}
Let $\lambda\geq 0$ and define $A(\lambda)=\{x\in X\mid \phi(x)\leq
\lambda\}$.   We will show that $A(\lambda)$ is invariant under $T$.

By the commutation relation for $\phi$ and $\tilde{T_t}$ we have
\begin{align*}
\phi \tilde{T_t} \1_{A(\lambda)}=\tilde{T_t} (\phi
\1_{A(\lambda)})\leq \lambda \tilde{T_t}  \1_{A(\lambda)},
\end{align*}
hence $(\phi-\lambda)\tilde{T_t}  \1_{A(\lambda)}\leq 0$ for all
$t\geq 0$. For $y\in A(\lambda)^c$ this implies
\begin{align*}
0\geq \underbrace{(\phi(y)-\lambda)}_{>0}(\tilde{T_t}
\1_{A(\lambda)})(y),
\end{align*}
so $\tilde{T_t}  \1_{A(\lambda)}(y)\leq 0$ for all $t\geq 0$. On the
other hand, $0\leq \1_{A(\lambda)}\leq 1$ implies $\tilde{T_t}
\1_{A(\lambda)}\leq 1$. Put together, we can conclude $\tilde{T_t}
\1_{A(\lambda)}\leq \1_{A(\lambda)}$ for all $t\geq 0$, that is,
$A(\lambda)$ is invariant.

Now obviously,
 $$A(\lambda)\subset A(\gamma) \; \mbox{ for }\; \lambda\leq
\gamma, \; X = \bigcup_{\lambda\geq 0}A(\lambda)\;  \mbox{ and }\;
\bigcap_{\lambda\geq 0}A(\lambda)^c = \emptyset.$$

Thus, since $(T_t)$ is irreducible, there is a (unique) $\beta\geq
0$ such that $m(A(\lambda))=0$ for $\lambda<\beta$ and
$m(A(\lambda)^c)=0$ for $\lambda>\beta$. It follows easily that
$\phi=\beta$ almost everywhere.
\end{proof}
\begin{remark} The  commutation property given in the lemma gives in
fact a characterization of irreducibility, as can bee seen from the characterization of invariant sets in \cite{Sch04}, Theorems 6 and 7.
\end{remark}

\begin{theorem}[Main properties of intertwining order isomorphisms]\label{U_unitary}
Let $p\in[1,\infty)$ and let $U\colon L^p(X_1,m_1)\lra L^p(X_2,m_2)$
be an order isomorphism intertwining irreducible Markovian
semigroups $(T_t^{(1)})$ and $(T_t^{(2)})$. Then there is a constant
$\beta>0$ such that $\widetilde U\widetilde{U^\ast}=\beta$,
$\widetilde{U^\ast}\widetilde U=\beta$ and
\begin{align*}
\tau_\#(h^2m_2)=\beta m_1
\end{align*}
for the transformation $\tau$ associated with $U$.  Moreover, if
$p\neq 2$, $h$ and $h^{-1}$ are bounded, and for every
$r\in[1,\infty]$ the restriction $U|_{L^p\cap L^r}$ extends to an
order isomorphism $U^{(r)}\colon L^r(X_1,m_1)\lra L^r(X_2,m_2)$
intertwining $(T_t^{(1)})$ and $(T_t^{(2)})$ with the same scaling
and transformation as $U$.
\end{theorem}
\begin{proof}
From Lemma \ref{adjoint_U} it follows easily that for $f \in
L_+(X_1,m_1)$ we have
$$\widetilde{U^\ast}\tilde U f = \phi f$$
with $\phi=\frac{d(\tau_\# m_2)}{dm_1}(h\circ\tau^{-1})^2$.

Denote by $q$ the dual exponent of $p$. Since $U$ is an order
isomorphism intertwining $(T_t^{(1)})$ and $(T_t^{(2)})$ on $L^p$
and $ U^\ast$ is an order isomorphism intertwining $(T_t^{(2)})$ and
$(T_t^{(1)})$ on $L^q$, we have $\tilde U \tilde T_t^{(1)}=\tilde
T_t^{(2)}\tilde U$ and $\widetilde{U^\ast}\tilde T_t^{(2)}=\tilde
T_t^{(1)}\widetilde{U^\ast}$. This gives $$\widetilde{U^\ast}\tilde
U \tilde T^{(1)}_t=\tilde T_t^{(1)}\widetilde{U^\ast}\tilde U.$$
Putting the last two displayed inequalities together we find
$$\phi T^{(1)}_t = T^{(2)}_t \phi$$
for all $t\geq 0$. By Lemma \ref{commutant_semigroup} there exists
then  a $\beta>0$ such that $\phi=\beta$ almost everywhere, that is,
$\beta m_1=\tau_\#(h^2m_2)$. This implies
\begin{align*}
\norm{Uf}_{L^p(X_2,m_2)}^p&=\int_{X_2}h^p\abs{f\circ\tau}^p\,dm_2\\
&=\beta\int_{X_2}h^{p-2}\abs{f\circ\tau}^p\,d(\tau^{-1}_\# m_1)\\
&=\beta\int_{X_1}(h\circ\tau^{-1})^{p-2} \abs{f}^p\,dm_1.
\end{align*}
Since $U$ is bounded, the multiplication operator
$M_{(h\circ\tau^{-1})^{p-2}}$ is bounded on $L^p(X_1, m_1)$ and so
$h^{p-2}$ is bounded. The same argument for $U^{-1}$ yields that
$h^{2-p}$ is also bounded. Hence, if $p \neq 2$, both $h$ and
$h^{-1}$ are bounded.

Now it is easily seen that $U|_{L^p\cap L^r}$ is bounded for any $r\in[1,\infty]$ and that the extension to $L^r(X_1,m_1)$ is given by
\begin{align*}
U^{(r)}\colon L^r(X_1,m_1)\lra L^r(X_2,m_2),\,U^{(r)}f=h\cdot (f\circ\tau).
\end{align*}
In particular, $U^{(r)}=\tilde U|_{L^r}$ and so $\tilde U\tilde T_t^{(1)}=\tilde T_t^{(2)}\tilde U$ implies that $U^{(r)}T_t^{(1)}=T_t^{(2)}U^{(r)}$.
\end{proof}

As a direct consequence of Theorem \ref{U_unitary} and the equality
$\|U U^\ast\| = \|U\|^2 = \norm{U^\ast U}$ for $U : L^2
(X_1,m_1)\longrightarrow L^2 (X_2, m_2)$,  we get the following
result in the case $p=2$.

\begin{corollary}[Intertwining order isomorphisms on $L^2$ are (almost) unitary] \label{transformation_measures}
Let $U\colon L^2(X_1,m_1)\lra L^2(X_2,m_2)$ be an order isomorphism intertwining the irreducible Markovian semigroups $(T_t^{(1)})$ and $(T_t^{(2)})$. Then $\frac 1{\norm{U}}U$ is unitary and
\begin{align*}
\tau_\#(h^2m_2)=\norm{U}^2m_1.
\end{align*}
\end{corollary}

\begin{remark}
The previous theorem  justifies why we restrict our attention to
intertwining on $L^2$.  It  shows that
intertwining by order isomorphisms of irreducible semigroups on any
$L^p$-space yields intertwining by order isomorphisms  on $L^2$.
Moreover, the corollary then shows that intertwining order
isomorphisms  on $L^2$ are actually (almost)  unitary.
\end{remark}

In the following sections we will primarily deal with the associated
quadratic forms instead of the semigroup itself. A strongly
continuous symmetric contraction semigroup on $L^2$ is Markovian if
and only if the associated closed, densely defined quadratic form
$Q$ is a \emph{Dirichlet form}, that is, $u\wedge 1\in D(Q)$ and
$Q(u\wedge 1)\leq Q(u)$ for all $u\in D(Q)$ (where $f\wedge g$
denotes the minimum of $f$ and $g$).

In general, the Dirichlet form does not transform nicely under order
isomorphisms  intertwining  the associated Markovian semigroup since
it is not clear how $U$ behaves with respect to the Hilbert space
structure on $L^2$. However, in the irreducible case the situation
is much better as the intertwining operator is (almost) unitary.

\begin{corollary}\label{corollary:intertwined forms}
For $i \in \{1,2\}$ let $(T^{(i)}_t)$ be irreducible Markovian
semigroups on $L^2(X_i,m_i)$,   $Q_i$ the corresponding Dirichlet
forms, and $U$ an order isomorphism intertwining $(T^{(1)}_t)$ and
$(T^{(2)}_t)$.
Then $UD(Q_1)=D(Q_2)$ and
$$Q_2(Uf,Ug)=\norm{U}^2 Q_1(f,g)$$
for all $f,g\in D(Q_1)$.
\end{corollary}
\begin{proof} The operator  $V:=\frac{1}{\|U\|} U$ clearly
intertwines the semigroups and is unitary by the preceding
corollary. Given this the statement of the corollary follows by
standard arguments. Here, are the details:  The associated Dirichlet
form is derived from the semigroup via
$$
D(Q_i)=\{f\in L^2(X_i,m_i)\mid \lim_{t\to0}\frac 1 t\langle
f-T_t^{(i)}f,f\rangle\text{ exists}\},$$
$$
Q_i(f,g)=\lim_{t\to 0}\frac 1 t\langle f-T_t^{(i)}f,g\rangle.
$$
Hence, for all $f,g\in L^2(X_1,m_1)$ we have
\begin{align*}
\frac 1t\langle V f-T_t^{(2)} V f, V g\rangle=\frac 1 t\langle V
(f-T_t^{(1)}f),V g\rangle=\cdot\frac 1t\langle
f-T_t^{(1)}f,g\rangle.
\end{align*}
In particular, $f\in D(Q_1)$ if and only if $V f\in D(Q_2)$, and
$Q_2(V f,V g)= Q_1(f,g)$ for $f,g\in D(Q_1)$ and the desired
statements follow.
\end{proof}

It turns out that the scaling $h$ belongs to a special class of
functions  and this can be used to show that it must be constant
under an  additional assumption of recurrence. Details are discussed
in the remaining part of this section.

\smallskip

We first introduce the relevant class of functions.
\begin{definition}
Let $(T_t)$ be a Markovian semigroup on $L^2(X,m)$. A function $u\in
L_+(X,m)$ is called {\em $(T_t)$-excessive} if $ {T_t} u\leq u$ for
all $t\geq 0$.
\end{definition}

\begin{lemma}[$h$ as excessive function]\label{h_excessive}
Let $(T_t^{(i)})$ be Markovian semigroups on $L^2(X_i,m_i)$,
$i\in\{1,2\}$, and $U$ an order isomorphism intertwining
$(T_t^{(1)})$ and $(T_t^{(2)})$. Then the associated scaling $h$ is
$(T_t^{(2)})$-excessive.
\end{lemma}
\begin{proof}
Since $\tilde T_t^{(1)}1\leq 1$, we have
\begin{align*}
\tilde T_t^{(2)}h= \tilde T_t^{(2)}\tilde U1=\tilde U \tilde
T_t^{(1)}1\leq \tilde U1=h
\end{align*}
for all $t\geq 0$.
\end{proof}

We now turn to the additional assumption on the semigroup.  Let
$(T_t)$ be a Markovian semigroup on $L^2(X,m)$. For $f\in
L^2_+(X,m)$ define the integral $S_N f=\int_0^N T_tf\,dt$  in the
Bochner sense. Then $(S_N f)$ is a monotone increasing sequence in
$L^2_+(X,m)$ and therefore there exists a $Gf\in L_+(X,m)$ such that
$S_N f\nearrow Gf$. The operator $G$ is positivity preserving and
therefore extends to $G\colon L_+(X,m)\lra L_+(X,m)$. A Markovian
semigroup $(T_t)$ is called \emph{transient} if $Gf<\infty$ almost
everywhere for some $f\in L_+(X,m)$ with $f>0$ a.e. It is called
\emph{recurrent} if $m(\{0<Gf<\infty\})=0$ for all $f\in
L^1_+(X,m)$.
A Dirichlet form is called irreducible (resp. transient, recurrent)
if the associated Markovian semigroup is irreducible (resp.
transient, recurrent). Irreducibility of $Q$ is equivalent to the
following property (see \cite{FOT94}, Theorem 1.6.1): If $A$ is a
measurable subset of $X$ with $\1_A f\in D(Q)$ and
\begin{align*}
Q(f)=Q(\1_A f)+Q(1_{A^c} f)
\end{align*}
for all $f\in D(Q)$, then $m(A)=0$ or $m(A^c)=0$. Recurrence of $Q$ is equivalent to $1\in D(Q)_e$ and $Q(1)=0$ (see \cite{FOT94}, Theorem 1.6.3).


By a standard result, an irreducible Markovian semigroup is either
recurrent or transient (see e.g. \cite{FOT94}, Lemma 1.6.4). In this
case, recurrence and transience can be characterized by a
Liouville-type property, namely the (non-) existence of non-constant
excessive functions. A convenient formulation for this is given in
the following result (see \cite{Kaj17}, Theorem 1).

\begin{lemma}\label{nexc-implies-rec}
Let $(T_t)$ be an irreducible Markovian semigroup. Then $(T_t)$ is
recurrent if and only if every $(T_t)$-excessive function is a.e.
constant.
\end{lemma}

Putting Lemma \ref{h_excessive} and Lemma \ref{nexc-implies-rec}
together we obtain the following rigidity statement for recurrent
semigroups.  For the special case of graphs this result is already
known and in fact one of the main achievements of  \cite{KLSW15}.

\begin{corollary}\label{recurrent_h_const}
Let $(T_t^{(i)})$, $i\in\{1,2\}$, be irreducible, recurrent
Markovian semigroups on $L^2(X_i,m_i)$. If $U\colon L^2(X_1,m_1)\lra L^2(X_2,m_2)$ is an order
isomorphism intertwining $(T_t^{(1)})$ and $(T_t^{(2)})$, then the
associated scaling $h$ is a.e. constant and, in particular, there
exists $\alpha >0$ with
$$\tau_\sharp m_2 = \alpha m_1.$$
\end{corollary}

\begin{remark} As the preceding considerations show  the scaling function is
constant in the  recurrent situation. Thus, it may be worthwhile to
point out that in our setting in general non-trivial scaling can not
be avoided. To see this we consider an  arbitrary  irreducible
Markovian semigroup $T_t$ on $L^2 (X,m)$ admitting an non-trivial
excessive function $h$. Then it is not hard to see that  $h$ must be
strictly positive and that the semigroup $T^{(2)}_t :=
M_{\frac{1}{h}} T_t M_h$ on $L^2 (X, h^2 m)$ is also Markovian,
where $M_g$ denotes the operator by multiplication with $g$.  Now,
clearly the semigroups $T_t$ and $T_t^{(2)}$ are intertwined by $U :
L^2 (X,m) \longrightarrow L^2 (X, h^2 m), U f = \frac{1}{h} f$.
\end{remark}

\begin{remark}
The considerations of this section can easily be carried over to
other families of Markovian operators such as semigroups of
Markovian operators over the natural numbers.
\end{remark}

\section{Regularity properties}\label{sec-Regularity}
In this section we study regularity properties of the scaling $h$
and transformation $\tau$ when the Dirichlet forms are not defined
merely on measure spaces, but on topological spaces. For that
purpose we need some compatibility of the Dirichlet form and the
underlying topology. We achieve this by working with quasi-regular
forms. Indeed, for our purposes the setting of quasi-regular forms
is not more involved than the -- maybe more common -- framework of
regular Dirichlet forms. At the same time it  offers the advantage that
we can deal with topological spaces  without local compactness
features. In this setting there exist natural replacements of the
concepts  of continuity and homeomorphism viz quasi-continuity and
quasi-homeomorphisms.  The main result of this section then shows
that the transformation associated with an intertwining order
isomorphism  is a quasi-homeomorphism (Theorem
\ref{tau_quasi-homeo}).   Along the way of
proving this result we have to develop the theory of  general nests
for the analytic capacity. This may well be of use in other
situations as well.

\medskip

First we recall some basic notions from the potential theory of
Dirichlet forms. For a comprehensive treatment see \cite{FOT94},
Chapter 2, and \cite{MR92}, Chapter III.  Since there are several
slightly different definitions for various objects of relevance in
potential theory in the literature,  we give an (almost)
comprehensive list of definitions and comment on subtleties.

Throughout this section let $X,X_1,X_2$ be Polish spaces and let
$m,m_1,m_2$ be $\sigma$-finite Borel measure of full support on
$X,X_1,X_2$. In particular, this assumption ensures that the arising
measure spaces are  standard Borel spaces.

Let $Q$ be a Dirichlet form on $X$. The {\em form norm}
$\|\cdot\|_Q$ on $D(Q)$ is given by
$$\|f\|_Q = \left(Q(f) + \|f\|_2^2\right)^{1/2}.$$

For $\phi\in L^2(X,m)$ with $\phi>0$ a.e. let $\psi=(L+1)^{-1}\phi$.
The {\em capacity} is defined as
\begin{align*}
\Capty_\psi(O)=\inf\{\norm{f}_Q^2\mid f\1_O\geq \psi\1_O\text{ a.e.}\}
\end{align*}
for $O\subset X$ open, and by
$$\Capty_\psi(E)=\inf \{ \Capty_\psi(O) \mid O\supset E\text{ open}\}$$
for arbitrary $E\subset X$. Since $\psi$ is nonnegative and belongs to $D(Q)$, the capacity is always well-defined.  We omit the index $\psi$ whenever the choice does not matter.

An ascending sequence $(G_k)_{k\in\IN}$ of subsets of $X$ is called
a \emph{nest} if
$$\lim_{k\to\infty}\Capty(G_k^c)=0.$$

From the subadditivity of the capacity (see \cite{MR92}, Theorem
2.8) it follows that for nests $(F_k)$ and $(G_k)$ the
\emph{refinement} $(F_k\cap G_k)$ is also a nest. Notice that we do
not demand the sets $G_k$ to be closed as is usually done.
Nevertheless, by the  definition of $\Capty$ and its monotonicity,
for each nest $(G_k)$ there exists a nest of closed sets $(F_k)$
with $F_k \subset G_k$.

\begin{remark}
So far only nests of closed sets seem to have been considered
in the setting of Dirichlet forms on topological spaces. In the
context of Dirichlet forms on measure spaces rater general nests
have already been studied in \cite{AH05,Sch16b}. However, let us
stress that our definition of a nest does not coincide with the ones
given there. This is due to the fact that we are in a topological
setting and work with the topological (analytic) capacity instead of
the measure theoretic one.
\end{remark}

For a subset $G$ of $X$ let
\begin{align*}
D(Q)_G=\{f\in D(Q)\mid \text{there exists }F\subset G\text{ closed
with  }f\1_{F^c}=0\text{ a.e.}\}.
\end{align*}
Hence, a function $f \in D(Q)$ belongs to $D(Q)_G$ if and only if
its measure theoretic support, i.e. the support of the measure $fm$,
is contained in $G$.

The following lemma characterizes nests in terms of the density of
functions vanishing outside the nest. It is a slight extension of
\cite{MR92}, Theorem 2.11, which only treats nests of closed sets.

\begin{lemma} \label{lemma:characterization nests in terms of capacity}
 An ascending sequence  $(G_k)$ of subsets of $X$ is a nest if and only if $\bigcup_{k\in \IN} D(Q)_{G_k}$ is dense in $D(Q)$ with respect to $\|\cdot\|_Q$.
\end{lemma}

\begin{proof} For a nest of closed sets $(F_k)$ the density of $\bigcup_{k \in \IN} D(Q)_{F_k}$ in $D(Q)$ follows from \cite{MR92}, Theorem 2.11. As remarked above, for an arbitrary nest $(G_k)$ there exists a nest of closed sets $F_k$ with $F_k \subset G_k$. This inclusion implies
 $$\bigcup_{k \in \IN} D(Q)_{F_k} \subset \bigcup_{k\in \IN} D(Q)_{G_k}$$
and the desired density follows from the statement for $\bigcup_{k \in \IN} D(Q)_{F_k}$.

Let $(G_k)$ be an ascending sequence of subsets of $X$ such that
$\bigcup_{k\in \IN} D(Q)_{G_k}$ is dense in $D(Q)$ with respect to
$\|\cdot\|_Q$ and let $\psi \in D(Q)$ be the function that appears
in the definition of $\Capty$. There then exists a sequence $\psi_n
\in \bigcup_{k\in \IN} D(Q)_{G_k}$ with $\psi_n \to \psi$ with
respect to $\|\cdot\|_Q$. According to our definition of $D(Q)_G$,
for each $n$ there is $k_n \in \IN$ and a closed set $F_{k_n}
\subset G_{k_n}$ with $\psi_n \1_{ F_{k_n}^c} = 0$. Moreover, it
follows from \cite{Sch16b}, Lemma~2.72 that
$$\Capty(O) = \inf\{\|\psi - f\|_Q^2 \mid f \1_O = 0 \text{ a.e.}\}$$
for every open $O \subset X$. With this observation, the choice of $(\psi_n)$ implies
\begin{align*}
\lim_{k \to \infty} \Capty_\psi(G_k^c) &= \inf_k \Capty_\psi(G_k^c) \leq \inf_n \Capty_\psi(F_{k_n}^c) \leq \inf_n \|\psi - \psi_n\|^2_Q = 0,
\end{align*}
showing that $(G_k)$ is a nest.
\end{proof}

A set $N\subset X$ is called \emph{polar} or {\em $Q$-exceptional}
if there is a nest $(G_k)_{k\in\IN}$ such that $N\subset \bigcap_k
G_k^c$.  Then, it is easy to see that  a set $N\subset X$ is polar
if and only if $\Capty(N)=0$. A pointwise property is said to hold
{\em quasi-everywhere} (q.e. for short) if it holds for all points
outside a polar set.

For a nest $(G_k)$ let
\begin{align*}
C(\{G_k\})=\{u\colon X\lra\IR\mid u|_{G_k}\text{ continuous for all }k\in\IN\}.
\end{align*}

A function $f\colon X\lra \IR$ is called \emph{quasi-continuous} if
there is a nest $(G_k)$ such that $f\in C(\{G_k\})$. If an a.e.
defined function $f$ has a quasi-continuous version, we write
$\tilde f$ for such a version (which is a.e. and q.e. unique). Note
that we also used a tilde to indicate extensions of positive
operators to $L_+(X,m)$.  We believe that no confusion should arise
from this conflict in notation.

Let $\hat X$ be a Polish space, $\hat m$ a $\sigma$-finite Borel
measure on $\hat X$ and $\hat Q$ a Dirichlet form on $L^2(\hat
X,\hat m)$. A map $\Phi\colon X\lra \hat X$ is called
\emph{quasi-homeomorphism} if there are nests $(G_k)$ in $X$, $(\hat
G_k)_{k\in\IN}$ in $\hat X$ of closed sets such that $\Phi\colon G_k\lra\hat G_k$
is a homeomorphism for all $k\in\IN$.

A Dirichlet form $Q$ on a locally compact space Polish space $X$ is
called \textit{regular} if $C_c (X)\cap D(Q)$ is dense in $C_c (X)$
with respect to the supremum norm and in $D(Q)$ with respect to the
form norm. (Here, $C_c (X)$ denotes the set of continuous functions
on $X$ with compact support). A generalization to our setting of --
not necessarily locally compact -- Polish spaces $X$  is given by
quasi-regular Dirchlet forms. Here, a Dirichlet form $Q$ is called
\emph{quasi-regular} if
\begin{itemize}
\item there exists a nest of compact sets,
\item there exists a dense subset of $(D(Q),\norm\cdot_Q)$ whose elements have qua\-si-con\-tin\-u\-ous versions,
\item there exist $f_n\in D(Q)$, $n\in\IN$, with quasi-continuous versions $\tilde f_n$ and a polar set $N$ such that $\{\tilde f_n\mid n\in\IN\}$ separates points of $X\setminus N$.
\end{itemize}

The connection between regular and quasi-regular forms is given by
the following characterization (see \cite{CMR94}, Theorem 3.7): A
Dirichlet form $Q$ is quasi-regular if and only if there exists  a
locally compact Polish space $\hat X$, a Radon measure $\hat m$ of
full support on $\hat X$, a regular Dirichlet form $\hat Q$ on
$L^2(\hat X,\hat m)$ and a quasi-homeomorphism $\Phi\colon X\lra
\hat X$ such that $\Phi_\#m=\hat m$ and $Q(f\circ\Phi)=\hat Q(f)$
for all $f\in L^2(\hat X,\hat m)$.

A set $F\subset X$ is called \emph{quasi-open} (resp.
\emph{quasi-closed}) if there exists a nest $(F_k)_{k\in\IN}$ of
closed sets such that $F\cap F_k$ is open (resp. closed) in $F_k$.
\begin{lemma}\label{preimage_quasi-open}
Let $Q$ be a quasi-regular Dirichlet form and $f\colon X\lra\IR$ quasi-continuous. Then preimages of open (resp. closed) sets under $f$ are quasi-open (resp. quasi-closed).
\end{lemma}
\begin{proof}
Let $(F_k)$ be a nest of closed sets such that $f\in C(\{F_k\})$, and let $A\subset \IR$ be open (resp. closed). Then $f ^{-1}(A)\cap F_k$ is open (resp. closed) in $F_k$ as the preimage of an open (resp. closed) set under a continuous map. Thus, $f^{-1}(A)$ is quasi-open (resp. quasi-closed).
\end{proof}

\begin{lemma}\label{images_quasi-homeo}
Let $\Phi\colon X\lra \hat X$ be a quasi-homeomorphism that maps nests to nests. Then images of quasi-open (resp. quasi-closed) sets are quasi-open (resp. quasi-closed).
\end{lemma}
\begin{proof}
Let $\Phi\colon X\lra \hat X$ be a quasi-homeomorphism and $A\subset X$ quasi-open. Let $(F_k)$ be a nest of closed sets in $X$ such that $\Phi(F_k)$ is closed and $\Phi\colon F_k\lra\Phi(F_k)$ is a homeomorphism for all $k\in \IN$, and let $(A_k)$ be a nest of closed subsets of $X$ such that $A\cap A_k$ is open in $A_k$ for all $k\in\IN$.

Then $(A_k\cap F_k)$ is a nest, and so is $(\Phi(A_k\cap F_k))$ by assumption. Since $\Phi$ is a homeomorphism on $F_k$ and $\Phi(F_k)$ is closed, the set $\Phi(A_k\cap F_k)$ is closed. Moreover, $\Phi(A)\cap \Phi(A_k\cap F_k)=\Phi(A\cap A_k\cap F_k)$ is open in $\Phi(A_k\cap F_k)$ as the image of an open set under a homeomorphism.

Of course, the proof for quasi-closed sets works analogously.
\end{proof}

Let $(G_k)$ be a nest. We say that $f\colon X\lra\IR$ is in the
\textit{local  space} of $(G_k)$ if for all $k \in \IN$ there exists
an $f_k\in D(Q)$ with $f|_{G_k}=f_k|_{G_k}$ a.e. We write
$D_\loc(\{G_k\})$ for the space of all functions in the local space
of $(G_k)$ and
\begin{align*}
D(Q)_\loc^\bullet =\bigcup D_\loc(\{G_k\}),
\end{align*}
where the union is taken over all nests $(G_k)$ of quasi-open sets.
This definition of the local space is taken from \cite{Kuw98},
Section 4.

\begin{lemma}\label{approx_D_loc}
Let $Q$ be a quasi-regular Dirichlet form, $f\in D(Q)_\loc^\bullet$ and $D\subset D(Q)$ a dense subspace. Then $f$ has a quasi-continuous version $\tilde f$, and there exists a sequence $(f_n)$ in $D$ such that $\tilde f_n\to \tilde f$ q.e.
\end{lemma}
\begin{proof}
That $f$ has a quasi-continuous version is the content of
\cite{Kuw98}, Lemma 4.1. Hence there is a nest $(G_k) $ of
quasi-open subsets and a sequence $(g_k)$ in $D(Q)$ such that
$\tilde f|_{G_k}=\tilde g_k|_{G_k}$ q.e. Let $G_k'=\{x\in G_k\mid
\tilde f(x)=\tilde g_k(x)\}$.

Since $D$ is dense in $D(Q)$, it follows from \cite{MR92},
Proposition III.3.5 that for every $n\in\IN$ there is a closed set
$F_n\subset X$ with $\Capty(F_n^c)<2^{-n}$ and an $f_n\in D$ such
that
\begin{align*}
\lim_{n\to\infty}\sup_{x\in F_n}\abs{\tilde f_n(x)-\tilde g_n(x)}= 0.
\end{align*}

Let $F_k'=\bigcap_{n=k}^\infty F_k$ for $k\in\IN$. Then $(F_k')$ is ascending and
\begin{align*}
\Capty(X\setminus F_k')=\Capty\left(\bigcup_{n=k}^\infty F_k^c\right)\leq\sum_{n=k}^\infty 2^{-n}\to 0,\,k\to\infty.
\end{align*}

Hence $(F_k'\cap G_k')$ is a nest. For all $x\in F_k'\cap G_k'$ and $n\geq k$ we have
\begin{align*}
\abs{\tilde f(x)-\tilde f_n(x)}\leq \sup_{y\in F_n'\cap G_n'}\abs{\tilde f(y)-\tilde f_n(y)}=\sup_{y\in F_n'\cap G_n'}\abs{\tilde f_n(y)-\tilde g_n(y)}\to 0,\,n\to \infty.
\end{align*}
Since $X\setminus \bigcup_k(F_k'\cap G_k')$ is polar, $(\tilde f_n)$ converges q.e. to $\tilde f$.
\end{proof}

The following lemma is  the technical key to establishing regularity
properties for excessive functions.

\begin{lemma}[Main tool]\label{char_excessive}
Let $(T_t)$ be a Markovian semigroup and $Q$ the associated Dirichlet form. If $h\in L_+(X,m)$ is $(T_t)$-excessive and $f\in D(Q)$, then $f\wedge h,(f-h)_+\in D(Q)$ and $Q(f\wedge h)\leq Q(f)$, $Q((f-h)_+)\leq 4Q(f)$.
\end{lemma}
\begin{proof}
The statement about $f\wedge h$ is the  content of \cite{Kaj17},
Proposition 4. As for the statement about $(f-h)_+$, note that
$(f-h)_+=f-f\wedge h$. Thus
\begin{flalign*}
&&Q((f-h)_+)=Q(f-f\wedge h)\leq 2 Q(f)+2Q(f\wedge h)\leq 4 Q(f).&&\qedhere
\end{flalign*}
\end{proof}

\begin{remark}
With a proof along the lines of \cite{Sch16b}, Lemma 2.50 and Theorem 2.57, one can show $Q((f-h)_+)\leq Q(f)$, but the weaker estimate from the lemma is sufficient for our purposes.
\end{remark}

\begin{lemma}\label{h_local}
Let $(T_t)$ be a Markovian semigroup and assume that the associated
Dirichlet form $Q$ is quasi-regular. If $h$ is $(T_t)$-excessive,
then there is a nest $(G_k)$ of quasi-open sets such that $h\wedge
M\in D_\loc(\{G_k\})$ for all $M\geq 0$.
\end{lemma}
\begin{proof}
By \cite{Kuw98}, Theorem 4.1, there is a nest $(G_k)$ of quasi-open
sets and a sequence $(f_k)\in D(Q)$ with $f_k|_{G_k}=1$ a.e. for all
$k\in\IN$. According to Lemma \ref{char_excessive}, $(Mf_k)\wedge
h\in D(Q)$. Now $(Mf_k\wedge h)|_{G_k}=(h\wedge M)|_{G_k}$, hence
$h\wedge M\in D_\loc(\{G_k\})$.
\end{proof}

\begin{proposition}[Excessive functions contained in local space]\label{h_local_strong}
Let $(T_t)$ be a Markovian semigroup and assume that the associated Dirichlet form $Q$ is quasi-regular. If $h$ is $(T_t)$-excessive, then $h\in D(Q)^\bullet_\loc$.
\end{proposition}
\begin{proof}

By Lemma \ref{h_local} and \cite{Kuw98}, Lemma 4.1 there is a
quasi-continuous version $\widetilde{h\wedge n}$ of $h\wedge n$ for
all $n\in\IN$. We first show that $(H_n)=(\{\widetilde{h\wedge
2n}\leq n\})$ is a nest.

The sequence $(H_n)$ consists of quasi-closed sets, see
Lemma~\ref{preimage_quasi-open}. Therefore, there exists a nest of
closed sets $(F_n)$ such that for each $n$ the set $H_n \cap F_n$ is
closed. By Lemma~\ref{lemma:characterization nests in terms of
capacity} and since $(F_n)$ is a nest, it suffices to prove that
$\bigcup_{n \in \IN} D(Q)_{H_n \cap F_n}$ is dense in $\bigcup_{n
\in \IN} D(Q)_{F_n}$.

To this end, let $k \in \IN$ and $f\in D(Q)_{F_k}$ with $f\geq 0$.
Since $h$ is $(T_t)$-excessive, so is $\frac{h}{n}\wedge n$. Let
$f_n:=(f\wedge n-\frac{h}{n}\wedge n)_+$. Obviously, $f_n=0$ a.e. on
$H_{n^2}^c$ and  $|f_n| \leq |f| =  0$ a.e. on $F_k^c$. For $n^2
\geq k$ this shows $f_n = 0$ a.e. on $(H_{n^2} \cap F_{n^2})^c$.  By
Lemma \ref{char_excessive} we have $f_n\in D(Q)$, and $H_{n^2} \cap
F_{n^2}$ is closed. Therefore, $f_n \in \bigcup_{k \in \IN}
D(Q)_{H_k \cap F_k}$ for all $n \in \IN$.

According to Lemma \ref{char_excessive},  $f_n\in D(Q)$ and
$Q(f_n)\leq 4 Q(f\wedge n)\leq 4 Q(f)$. Thus, every subsequence of $(f_n)$ has a weakly
convergent subsequence in $(D(Q),\norm\cdot_Q)$. Since $f_n\to f$ in
$L^2$, the limit is $f$. Hence $f_n\to f$ weakly in
$(D(Q),\norm\cdot_Q)$. Weak closures and strong closures of convex
sets agree in Hilbert spaces. Therefore, $f$ belongs to the  closure
of $\bigcup_{n \in \IN} D(Q)_{H_n \cap F_n}$ in
$(D(Q),\norm\cdot_Q)$ and we arrive at the conclusion that $(H_n)$
is a nest.

By Lemma \ref{preimage_quasi-open} the set $G_n=\{\widetilde{h\wedge
2n}<n+1\}$ is quasi-open.  Since $H_n\subset G_n$,  $(G_n)$ is  a
nest.  Let $(\hat G_n)$ be a nest as in Lemma \ref{h_local}. Then
the refinement $(G_n \cap \hat G_n)$ is also a nest. For $n\in\IN$
let $f_{n}\in D(Q)$ such that $f_{n}|_{\hat G_n} =h\wedge
(n+1)|_{\hat G_n}$ a.e. By the definition of $G_n$ we have $(h\wedge
(n+1))|_{G_n}=h|_{G_n}$ a.e. Thus, $f_{n}|_{G_n \cap \hat
G_n}|=h|_{G_n\cap \hat G_n}$ a.e. and we arrive at $h\in D_{\rm
loc}(\{G_k \cap \hat G_k\})$. Since $G_n$ and $\hat G_n$ are
quasi-open, so is their intersection and we obtain $h\in
D(Q)^\bullet_\loc$.
\end{proof}

\begin{remark}
We were not able to find the  result  above  in the
literature. For a strongly local regular Dirichlet form it is proven
in \cite{Stu94}, Lemma~3, that locally bounded excessive function
belong to the local space (with respect to a nest of compact sets).
\end{remark}

\begin{lemma}\label{comp_tau}
Let $Q_i$ be irreducible quasi-regular Dirichlet forms on $L^2(X_i,m_i)$, $i \in \{1,2\}$, and let $U:L^2(X_1,m_1) \longrightarrow  L^2(X_2,m_2)$ be an order isomorphism intertwining the associated semigroups. Denote by $\tau$ the associated transformation. Then there is a nest $(G_k)_{k\in\IN}$ of quasi-open subsets of $X_1$ such that $f\circ\tau\in D(Q_2)$ for all $f\in \bigcup_k D(Q_1)_{G_k}\cap L^\infty(X_1,m_1)$.
\end{lemma}
\begin{proof}
Denote by $H$ the scaling of $U^{-1}$. We have $H \circ \tau = 1/h,$ where $h$ is the scaling of $U$. By Lemma \ref{h_excessive} and Proposition \ref{h_local_strong} it satisfies $H\in D(Q_1)_\loc^\bullet$. By \cite{Kuw98}, Theorem 4.1, there is a nest $(G_k)_{k\in\IN}$ of quasi-open subsets of $X_1$ and a sequence $(H_k)_{k\in\IN}$ in $D(Q_1)\cap L^\infty(X_1,m_1)$ such that $H|_{G_k}=H_k|_{G_k}$ a.e. for all $k\in\IN$.

Let $f\in D(Q_1)\cap L^\infty(X_1,m_1)$ with $f|_{G_k^c}=0$ a.e. Then
\begin{align*}
f\circ\tau=h(Hf)\circ\tau=h(H_kf)\circ\tau\quad\text{a.e.}
\end{align*}
Since $D(Q_1)\cap L^\infty(X_1,m_1)$ is an algebra (see \cite{FOT94}, Theorem 1.4.2), we have $H_k f\in D(Q_1)$ and therefore $f\circ\tau=U(H_k f)\in D(Q_2)$ by Corollary~\ref{corollary:intertwined forms}.
\end{proof}

\begin{lemma}\label{h=0_polar}
Let $Q_i$ be irreducible quasi-regular Dirichlet forms on $L^2(X_i,m_i)$, $i \in \{1,2\}$, and let $U:L^2(X_1,m_1) \longrightarrow  L^2(X_2,m_2)$ an order isomorphism intertwining the associated semigroups. Denote by $h$ and $\tau$ the associated scaling and transformation. Then $\{\tilde h=0\}$ is polar, $f\circ\tau$ has a quasi-continuous version and $\widetilde{Uf}=\tilde h\cdot\widetilde{f\circ\tau}$ q.e. for all $f\in D(Q_1)$.
\end{lemma}
\begin{proof}

Let $(G_k)$ be a nest as in Lemma \ref{comp_tau} and let $f\in
D(Q_1)_{G_k}\cap L^\infty(X_1,m_1)$ for some $k\in\IN$. Then
$f\circ\tau\in D(Q_2)$, hence it has a quasi-continuous version
$\widetilde{f\circ\tau}$. Since both $\widetilde{Uf}$ and $\tilde
h\cdot\widetilde{f\circ\tau}$ are versions of $Uf$, they coincide
q.e. (see \cite{FOT94}, Lemma~2.1.4). In particular,
$\widetilde{Uf}=0$ q.e. on $\{\tilde h=0\}$.

Since $(G_k)$ is a nest, $\bigcup_k D(Q_1)_{G_k}$ is dense in
$D(Q_1)$. By the cut-off property of Dirichlet forms, $(g\wedge
k)\vee (-k)\to g$ w.r.t $\norm\cdot_{Q_1}$ for all $g\in D(Q_1)$.
Thus $\bigcup_k D(Q_1)_{G_k}\cap L^\infty(X_1,m_1)$ is also dense in
$D(Q_1)$. Hence for every $g\in D(Q_2)$ there is a sequence $(f_n)$
in $\bigcup_k D(Q_1)_{G_k}\cap L^\infty(X_1,m_1)$ such that $Uf_n\to
g$ w.r.t. $\norm\cdot_{Q_2}$.

By \cite{MR92}, Proposition III.3.5 there is a subsequence
$(f_{n_j})$ such that $\widetilde{Uf_{n_j}}\to \tilde g$ q.e. Thus
$\tilde g=0$ q.e. on $\{\tilde h=0\}$ for all $g\in D(Q_2)$. Since
$Q_2$ is quasi-regular, there is a countable collection $\{g_n \mid
n\in \IN\} \subset D(Q_2)$ and a polar set $N \subset X_2$ such that
quasi-continuous versions $\{\tilde g_n \mid n \in \IN\}$ separate
the points of $X \setminus N$. Moreover, by the subadditivity of the
capacity there is another polar set $N'$ such that $\tilde g_n = 0$
on $\{\tilde h = 0\} \setminus N'$ for all $n \in \IN$. In
particular, $\{\tilde g_n \mid n \in \IN\}$ does not separate the
points of $\{\tilde h = 0\} \setminus N'$ and so $\{\tilde h = 0\} $
must be polar.
\end{proof}

In order to prove the main theorem of this section, we need to
recall the following regularity property for nests. A closed set $F
\subset X$ is called {\em  regular} if its measure theoretic support
satisfies ${\rm supp}(\1_F m) = F$. A nest of closed sets $(F_k)$ is
{\em regular} if for all $k \in \IN$ the set $F_k$ is regular. The
main merit of working with regular nests is that  for such nests the
concepts of quasi-everywhere and almost-everywhere are compatible in
the following sense: For a regular nest $(F_k)$ a function $f \in
C(\{F_k\})$ satisfies $f \geq 0$ a.e. if and only if $f(x) \geq 0$
for all $x \in \bigcup_k F_k$, see \cite{FOT94}, Theorem~2.1.2. Note
that for any nest of closed sets $(F_k)$ the sequence $({\rm
supp}(\1_{F_k} m))$ is a regular nest, see  \cite{FOT94},
Lemma~2.1.3. Thus, for most purposes nests of closed sets can be
assumed to be regular.

We begin with a lemma, which is a variant of \cite{FOT94}, Theorem 2.1.2.
\begin{lemma}\label{unique_quasi_homeo}
Let $Q_i$ be quasi-regular Dirichlet forms on
$L^2(X_i,m_i)$, $i \in \{1,2\}$. If $\Phi,\Phi'\colon X_2\lra X_1$ are quasi-homeomorphisms such that $\Phi=\Phi'$ a.e., then $\Phi=\Phi'$ q.e.
\end{lemma}
\begin{proof}
Let $(G_k)$, $(G_k')$ be nests such that $\Phi|_{G_k}$ and $\Phi'|_{G_k'}$ are continuous. Otherwise restricting to $\supp(\1_{G_k}m_2)$ and $\supp(\1_{G_k'}m_2)$ respectively, we can assume that the nests are regular. Then $(G_k\cap G_k')$ is also a regular nest, and thus $\Phi=\Phi'$ a.e. implies $\Phi|_{G_k\cap G_k'}=\Phi'|_{G_k\cap G_k'}$. Therefore, $\Phi=\Phi'$ q.e.
\end{proof}

\begin{theorem}[Regularity of the transformation]\label{tau_quasi-homeo}
Let $Q_i$ be irreducible quasi-regular Dirichlet forms on
$L^2(X_i,m_i)$, $i \in \{1,2\}$, and let $U:L^2(X_1,m_1)
\longrightarrow  L^2(X_2,m_2)$ be an order isomorphism intertwining
the associated semigroups. Then the transformation associated with
$U$ has a version $\tilde \tau$ that is a quasi-homeomorphism. Up to
equality q.e. this version is unique.
\end{theorem}
\begin{proof}
Since every quasi-regular Dirichlet form is quasi-homeomorphic to a
regular Dirichlet form and compositions of quasi-homeomorphisms are
quasi-ho\-meo\-mor\-phisms, it suffices to prove the assertion in
the regular case.

Since $X_1$ is assumed to be metrizable and separable and $Q_1$ is
regular, there is a countable dense subalgebra  $D\subset
C_0(X_1)\cap D(Q_1)$ that is uniformly dense in $C_0(X_1)$. Since
$UD\subset D(Q_2)$ and $Q_2$ is regular, there is a regular nest
$(F_k)$ of closed subsets of $X_2$ such that for every $f\in D$
there is a version $\widetilde{Uf}$ of $Uf$ such that
$\widetilde{Uf}\in C(\{F_k\})$, and $h$ has a version $\tilde h$
such that $\tilde h\in C(\{F_k\})$, see \cite{FOT94}, Theorem~2.1.2.
Since $\{\tilde h=0\}$ is polar by Lemma \ref{h=0_polar}, we may
additionally assume that $\tilde h > 0$ on $\bigcup_k F_k$. Refining
by a regular nest of compact sets, we can moreover assume the
$(F_k)$ to be compact.

Since $\abs{Uf}\leq h \norm{f}_\infty$ and $(Uf)(Ug)=h U(fg)$ a.e.
for all $f,g\in D$, the regularity of the nest $(F_k)$ implies
$\abs{\widetilde{Uf}}\leq \tilde{h} \norm{f}_\infty$ and
$(\widetilde{Uf})(\widetilde{Ug})=\tilde h \widetilde{U(fg)}$ on
$\bigcup_k F_k$ for all $f,g\in D$. Thus, for $y\in F_k$ the map
\begin{align*}
D\lra \IR,\,f\mapsto \frac 1{\tilde h(y)}\widetilde{Uf}(y)
\end{align*}
extends continuously to a multiplicative linear map $\chi_y$ on
$C_0(X_1)$. By Gelfand-Naimark theory, there exists a unique
$\tilde\tau(y)\in X_1$ such that $\chi_y(f)=f(\tilde\tau(y))$ for
all $f\in C_0(X_1)$.

We prove that $\tilde\tau$ is continuous on $F_k$. For if not, there
exists a $y\in F_k$, a neighborhood $V$ of $\tilde\tau(y)$ and a
sequence $(y_n)$ such that $y_n\to y$ and $\tilde\tau(y_n)\notin V$.
Let $f\in D$ with $f(\tilde\tau(y)) > 0$ and $\supp f\subset V$.
Then
\begin{align*}
\tilde h(y)
f(\tilde\tau(y))=\widetilde{Uf}(y)=\lim_{n\to\infty}\widetilde{Uf}(y_n)=\tilde
h(y_n) f(\tilde\tau(y_n))=0,
\end{align*}
a contradiction to $\tilde h(y) > 0$.

It is easy to see that the construction is consistent for different
$k$ so that $\widetilde{Uf}=\tilde h\cdot (f\circ\tilde \tau)$ on
$\bigcup_k F_k$ for all $f\in D$. Since $m_2(X \setminus \bigcup_k
F_k) = 0$ and $\tilde h > 0$ on $\bigcup_k F_k$,  we obtain $f \circ
\tilde \tau = f \circ \tau$ a.e. for all $f \in D$ and so $\tilde
\tau = \tau$ a.e.

Since $F_k$ is compact and $\tilde \tau$ is continuous on $F_k$, the
image $\tilde\tau(F_k)$ is compact. The measure $\tilde \tau_\# m_2
=  \tau_\# m_2$ is equivalent to $m_1$ and so $\tilde\tau(F_k)$ is
regular.  We prove that $(\tilde\tau(F_k))$ is a nest by showing $U
D(Q_1)_{\tilde\tau(F_k)}= D(Q_2)_{F_k}$. Since both  $F_k$ and
$\tilde\tau(F_k)$ are closed, Corollary~\ref{corollary:intertwined
forms} implies that it suffices to show  $(Uf) \1_{X_2 \setminus
F_k} = 0$ a.e. if and only if $f \1_{X_1 \setminus \tilde \tau
(F_k)} = 0$ a.e. This however is a consequence of
Corollary~\ref{transformation_measures} and the fact that $\tau =
\tilde \tau$ a.e.

The transformation associated with $U^{-1}$ is given by $\tau^{-1}$.
An application of the above arguments to $U^{-1}$ yields a regular
nest of compact sets $(F'_k)$ in $X_1$ and an $m_1$-version
$\widetilde{\tau^{-1}}$ of $\tau^{-1}$ that is continuous on $F'_k$
for all $k \in \IN$.

We let $G_k : =  F_k \cap \tilde \tau^{-1}(F'_k) = F_k \cap \tilde
\tau^{-1}(F'_k \cap \tilde \tau(F_k))$. Since $\tilde \tau$ is
continuous on $F_k$ and $F_k$ is compact, $G_k$ is compact.
Moreover, as a refinement of two regular nests $(\tilde \tau (G_k))
= (F'_k \cap \tilde \tau (F_k))$ is a regular nest.  With the same
arguments as for proving that $(\tilde \tau(F_k))$ is a regular
nest, it follows that $G_k$ is a regular nest.

Next we prove that  $\tilde \tau|_{G_k}$ is injective. We chose
$G_k$ such that the restriction of the composition
$\widetilde{\tau^{-1}} \circ \tilde \tau|_{G_k}$ is continuous.
Moreover, since $\tilde \tau$ is an $m_2$-version of $\tau$,
$\widetilde{\tau^{-1}}$ is an $m_1$-version of $\tau^{-1}$ and since
$\tau^{-1}_\# m_1$ and $m_2$ are equivalent, we have $
\widetilde{\tau^{-1}} \circ \tilde \tau = \tau^{-1} \circ \tau =
\id_{X_2}$ a.e. Thus, the regularity of $G_k$ and the continuity of
$\widetilde{\tau^{-1}} \circ \tilde \tau$ implies
$\widetilde{\tau^{-1}} \circ \tilde \tau|_{G_k} = \id_{X_2}$ proving
the injectivity of $\tilde \tau |_{G_k}$.

Since $\tilde \tau|_{G_k}$ is continuous, $G_k$ is compact and
$\tilde \tau|_{G_k} :G_k \to \tilde \tau(G_k)$ is bijective, it
follows that $\tilde \tau|_{G_k} :G_k \to \tilde \tau(G_k)$ is a
homeomorphism. As we have seen above $(G_k)$ and $(\tilde
\tau(G_k))$ are nests and so $\tilde{\tau}$ is a
quasi-homeomorphism.

Since the transformation $\tau$ is determined up to equality a.e.,
the q.e. uniqueness of $\tilde \tau$ follows from Lemma
\ref{unique_quasi_homeo}.
\end{proof}

In the situation of the previous theorem, we call $\tilde \tau$ simply the \emph{quasi-homeo\-morphism associated with $U$}. As result of the theorem, it is uniquely determined up to equality q.e.

Our final aim in this section it to show that intertwining  order
isomorphism preserve not only the form domain but also the local
spaces.  We will need the following technical feature of $\tilde
\tau, \tilde \tau^{-1}$.

\begin{proposition}\label{U_D_loc} Let $Q_i$ be irreducible quasi-regular Dirichlet forms on
$L^2(X_i,m_i)$, $i \in \{1,2\}$, and let $U:L^2(X_1,m_1)
\longrightarrow  L^2(X_2,m_2)$ an order isomorphism intertwining the
associated semigroups and let $\tilde \tau$ be the quasi-homeomorphism associated with $U$. Then $\tilde \tau, \tilde \tau^{-1}$
map nests to nests. Moreover, if the elements of the nest are
quasi-open, then so are the elements of its image under $\tilde
\tau, \tilde \tau^{-1}$.
\end{proposition}
\begin{proof}  Let $(F_k)$ be a nest of closed subsets of $X_1$. By definition
there are nests $(A_k)$, $(B_k)$ of closed subsets of $X_1$, $X_2$
such that $\tilde\tau\colon B_k\lra A_k$ is a homeomorphism. Then
$(A_k\cap F_k)$ is a nest in $X_1$, $\tilde\tau^{-1}(A_k\cap F_k)$
is closed in $X_2$ and $D(Q_2)_{\tilde\tau^{-1}(A_k\cap F_k)} =
UD(Q_1)_{A_k \cap F_k}$ (cf. proof of
Theorem~\ref{tau_quasi-homeo}). Hence $(\tilde\tau^{-1}(A_k\cap
F_k))$ is a nest in $X_2$ and so is $(\tilde\tau^{-1}(F_k))$. If
$(G_k)$ is a nest of not necessarily closed subsets of $X_1$, there
is a nest $(F_k)$ of closed sets such that $F_k\subset G_k$ by
definition. Thus $\tilde\tau^{-1}$ maps nests to nests (and of
course the same holds for $\tilde\tau$).

The last statement follows from the already shown part and Lemma
\ref{images_quasi-homeo}.
\end{proof}

The following lemma extends Corollary~\ref{corollary:intertwined
forms}.

\begin{lemma}[Order isomorphisms preserve local space]
\label{lemma:tau leaves local space invariant}
Let $Q_i$ be irreducible quasi-regular Dirichlet forms on
$L^2(X_i,m_i)$, $i \in \{1,2\}$, and let $U:L^2(X_1,m_1)
\longrightarrow  L^2(X_2,m_2)$ an order isomorphism intertwining the
associated semigroups. Denote by $h$ and $\tilde\tau$ the associated
scaling and quasi-homeomorphism. Let  $f\in
D(Q_1)_\loc^\bullet$ be arbitrary.  Then, both   $h\cdot
(f\circ\tau)$ and $ f\circ \tau $ belong to $D(Q_2)_\loc^\bullet$.
Moreover, $f$ and $f \circ \tau$ have quasi-continuous versions that
are related by $ \widetilde{f \circ \tau} = \tilde f \circ \tilde
\tau$ quasi everywhere.
\end{lemma}
\begin{proof}  We first show  $h\cdot (f\circ\tau) \in D(Q_2)_\loc^\bullet$ for
$f\in D(Q_1)_\loc^\bullet$.

Let now $(G_n)_{n\in\IN}$ be a nest of quasi-open sets and $(f_n)$ a
sequence in $D(Q_1)$ such that $f_n|_{G_n}=f|_{G_n}$ a.e.  Then
\begin{align*}
h\cdot (f\circ\tau)|_{\tilde\tau^{-1}(G_n)}=h\cdot
(f_n\circ\tau)|_{\tilde\tau^{-1}(G_n)} =
Uf_n|_{\tilde\tau^{-1}(G_n)}.
\end{align*}
As  the image of $(G_n)$ under $\tilde \tau^{-1}$ is a quasi-open
nest by the previous proposition,  we arrive at $h\cdot
(f\circ\tau)\in D(Q_2)_\loc^\bullet$ as desired.

We now turn to proving the remaining statements of (b). Here, we
improve the  arguments from Lemma \ref{comp_tau} with what we have
already shown. Let $H$ be the scaling $U^{-1}$ and let $f\in
D(Q_1)_\loc^\bullet$. Since $D(Q_1)_\loc^\bullet$ is an algebra by
\cite{Kuw98}, Theorem 4.1, and \cite{FOT94}, Theorem 1.4.2, we have
$Hf\in D(Q_1)_\loc^\bullet$.  By what we have already shown this
gives $f\circ\tau=h\cdot(Hf)\circ\tau\in D(Q_2)_\loc^\bullet$.

According to \cite{Kuw98}, Lemma~4.1, both $f$ and $f \circ \tau$
have quasi-continuous versions. Since $\tau = \tilde \tau$ a.e., it
follows that $\widetilde{f \circ \tau} = \tilde f \circ \tilde \tau$
a.e. and since $\tilde \tau$ is a quasi-homeomorphism that maps nets
to nests, $\tilde f \circ \tilde \tau$ is quasi-continuous.
Therefore, $\widetilde{f \circ \tau} = \tilde f \circ \tilde \tau$
by \cite{FOT94}, Theorem~2.1.2.
\end{proof}

\section{The Beurling-Deny decomposition}\label{sec-Beurling-Deny}
In this section we show how the Beurling-Deny decomposition of a
quasi-regular Dirichlet form transforms under an order isomorphism
(Theorem \ref{transformation_decomposition}).

\medskip

Let $Q$ be a quasi-regular Dirichlet form on $L^2(X,m)$. Let
$\Gamma$ be a measure on $X$ that charges no polar sets. A
quasi-closed set $F\subset X$ is called \emph{quasi-support} of
$\Gamma$ if $\Gamma(F^c)=0$ and for every quasi-closed set $\tilde
F$ with $\Gamma(\tilde F^c)=0$ the set $F\setminus \tilde F$ is
polar.

The quasi-support of $\Gamma$ is determined up to a polar set, and
we write $\supp_Q(\Gamma)$ whenever the choice does not matter. The
quasi-support of a measurable function $f\colon X\lra \IR$ is
defined as the quasi-support of $\abs{f} m$ and we denote it by
$\supp_Q(f)$.

Let $\Delta_X=\{(x,x)\mid x\in X\}$. The form $Q$ can be decomposed as

$$Q(f,g)=Q^{(c)}(f,g)+\int_{X\times X\setminus \Delta_X}(\tilde f(x)-\tilde f(y))(\tilde g(x)-\tilde g(y))\,dJ(x,y)\\
+\int_X \tilde f(x)\tilde g(x)\,dk(x)
$$
for all $f,g\in D(Q)$, where
\begin{itemize}
\item $Q^{(c)}$ is a positive symmetric bilinear form such that $Q^{(c)}(f,g)=0$ for all $f,g\in D(Q)$ such that $f$ is constant on a quasi-open set containing $\supp_Q(g)$,
\item $J$ is a $\sigma$-finite symmetric measure on $(X\times X)\setminus \Delta_X$ such that $J((N\times X)\setminus \Delta_X)=0$ for all polar $N\subset X$,
\item $k$ is a $\sigma$-finite measure on $X$ such that $k(N)=0$ for all polar $N\subset X$.
\end{itemize}
Moreover, $Q^{(c)}$, $J$ and $k$ are unique among all maps with the
properties listed above (see \cite{DMS97}, Theorem 1.2, or
\cite{Kuw98}, Theorem 5.1).

For $f \in D(Q) \cap L^\infty(X,m)$ the \textit{local part} of the
energy measure $\Gamma^{(c)}(f)$ is the measure defined by the
identity
\begin{align*}
\int_X \tilde{\varphi} \,d\Gamma^{(c)}(f)=Q^{(c)}(\varphi f,f)-\frac 1 2 Q^{(c)}(\varphi,f^2)
\end{align*}
for $\varphi\in D(Q)\cap L^\infty(X,m)$, see \cite{Kuw98}, Theorem 5.2. The local part of $Q$ then satisfies
\begin{align*}
Q^{(c)}(f)=\int_X d\Gamma^{(c)}(f)
\end{align*}
for all $f\in D(Q)\cap L^\infty(X,m)$. By polarization
$\Gamma^{(c)}(\cdot)$ can be extended to a measure valued bilinear
form $\Gamma^{(c)}(\cdot,\cdot)$ on $D(Q)\cap L^\infty(X,m)$. It
satisfies the following product rule (see \cite{Kuw98}, Lemma 5.2):
\begin{align*}
\Gamma^{(c)}(fg,h)=\tilde f\Gamma^{(c)}(g,h)+\tilde g\Gamma^{(c)}(f,h)
\end{align*}
for all $f,g,h\in D(Q)\cap L^\infty(X,m)$. Moreover, due to the
locality property of $Q^{(c)}$, the energy measure satisfies $\1_G
\Gamma^{(c)}(f)=0$ whenever $f$ is constant a.e. on the quasi-open
set $G$, see \cite{Kuw98}, Lemma~5.1. Thus $\Gamma^{(c)}$ can be
extended to $D(Q)_\loc^\bullet$ via
\begin{align*}
\1_{G_n}\Gamma^{(c)}(f)=\1_{G_n}\Gamma^{(c)}(f_n),
\end{align*}
where $(G_n)$ is a nest of quasi-open sets and $f_n\in D(Q)$ such that $\1_{G_n}f_n=\1_{G_n} f$ a.e. This extension still satisfies the Leibniz rule (see \cite{Kuw98},
Lemma~5.3 for details).  

%
%
%
%
%
%
Conversely, the product rule for the energy measure implies the
strong locality of the form. This result is essentially well-known,
see e.g. \cite{Stu94}. We state (and prove)  it here in a form
suitable for the proof of the main theorem of this section.

\begin{lemma}\label{Leibniz_strongly_local} Let $Q$ be a
quasi-regular Dirichlet form on $L^2 (X,m)$.  Denote by $M(X)$ the
signed measures on $X$ endowed with the total variation norm and let
\begin{align*}
\hat\Gamma\colon D(Q)\times D(Q)\lra M(X)
\end{align*}
be a positive, symmetric and bilinear map such that $\hat
\Gamma(f,f)(X)\leq Q(f)$ and $\hat\Gamma(f,f)$ charges no polar set
for all $f\in D(Q)$. Let $\hat Q(f,g)=\hat\Gamma(f,g)(X)$ for
$f,g\in D(Q)$.  If $D\subset D(Q)$ is a dense subspace and $\hat
\Gamma$ satisfies the product rule
\begin{align*}
\hat\Gamma(fg,h)=\tilde f\hat\Gamma(g,h)+\tilde g\hat\Gamma(f,h)
\end{align*}
for all $f,g,h\in D\cap L^\infty(X,m)$, then $\hat Q$ is strongly local in the sense that
\begin{align*}
\hat Q(f,g)=0
\end{align*}
whenever there exists an open set $V$ such that $\supp_Q g\subset V$ and $f$ is constant a.e. on $V$.
\end{lemma}
\begin{proof}
First we prove that the product rule holds indeed for all $f,g,h\in
D(Q)\cap L^\infty (X,m)$. Let $(f_k)$, $(g_k)$, $(h_k)$ be sequences
in $D\cap L^\infty(X,m)$ with $f_k\to f$, $g_k\to g$, $h_k\to h$ in
$D(Q)$. We may assume additionally that $\norm{f_k}_\infty\leq
\norm{f}_\infty$ for all $k\in \IN$, $\tilde f_k\to \tilde f$ q.e.
and the same for $(g_k)$, $(h_k)$.

By \cite{FOT94}, Theorem 1.4.2, we have $f_k g_k\to fg$ in $D(Q)$.
Since $\hat\Gamma$ is bounded, we have $\hat\Gamma(f_k g_k,h_k)\to
\hat\Gamma(fg,h)$ in total variation norm. On the other hand, the
product rule implies
\begin{align*}
\hat\Gamma(f_k g_k,h_k)=\tilde f_k \hat\Gamma(g_k,h_k)+\tilde g_k\hat\Gamma(f_k,h_k).
\end{align*}
Moreover,
\begin{align*}
\norm{\tilde f\hat\Gamma(g,h)-\tilde f_k\hat\Gamma(g_k,h_k)}&\leq \norm{(\tilde f-\tilde f_k)\hat\Gamma(g,h)}+\norm{\tilde f_k(\hat\Gamma(g,h)-\hat\Gamma(g_k,h_k))}\\
&\leq\int_X \abs{\tilde f-\tilde f_k}\,d\hat\Gamma(g,h)+\norm{f}_\infty \norm{\hat\Gamma(g,h)-\hat\Gamma(g_k,h_k)}.
\end{align*}
The first summand converges to $0$ by Lebesgue's theorem, while the
convergence of the second summand follows once again from the
boundedness of $\hat\Gamma$. The same argument can be applied to
$\tilde g_k\hat\Gamma(f_k,h_k)$ so that we arrive at
\begin{align*}
\hat\Gamma(fg,h)=\lim_{k\to\infty}(\tilde f_k\hat\Gamma(g_k,h_k)+\tilde g_k\hat\Gamma(f_k,h_k))=\tilde f\hat\Gamma(g,h)+\tilde g\hat\Gamma(f,h)
\end{align*}
as desired.

From the product rule we can infer $\hat\Gamma(f,f)(G)=0$ for every
quasi-opens set $G$ and every $f\in D(Q)$ such that $f$ is constant
a.e. on $G$. Indeed, this follows from \cite{FOT94}, Corollary
3.2.1,  by the transfer principle as noted in \cite{Kuw98}, Lemma
5.2. Notice that while these results are formulated for the strongly local energy measure, the
proofs only use the assumptions of the present lemma.

Now, if $V\subset X$ is open, $\supp_Q g\subset V$ and $f$ is constant a.e. on $V$, then
\begin{align*}
\abs{\hat \Gamma(f,g)}(X)&\leq \abs{\hat \Gamma(f,g)}(V)+\abs{\hat \Gamma(f,g)}(\supp_Q g)^c).
\end{align*}
As $f$ is constant a.e. on $V$ and $g=0$ a.e. on $(\supp_Q g)^c$, $\abs{\hat\Gamma(f,g)}(X)=0$ follows from the Cauchy-Schwarz inequality for $\hat\Gamma$, which also carries over easily to this setting. Thus, $\hat Q(f,g)=0$.
\end{proof}

For $f,\phi\in D(Q)\cap L^\infty(X,m)$ we define the truncated form
\begin{align*}
Q_\phi(f)=Q(\phi f)-Q(\phi f^2,\phi).
\end{align*}
From the product rule it follows that
$$
Q_\phi(f)=\int_X \tilde \phi^2\,d\Gamma^{(c)}(f)
\quad+\int_{(X\times
X)\setminus\Delta_X}\tilde\phi(x)\tilde\phi(y)\,(\tilde f(x)-\tilde
f(y))^2\,dJ(x,y).$$

\begin{theorem}[Transformation Beurling-Deny decomposition]\label{transformation_decomposition}
Let $Q_i$ be irreducible quasi-regular Dirichlet forms on
$L^2(X_i,m_i)$, $i \in \{1,2\}$, and let $U:L^2(X_1,m_1)
\longrightarrow  L^2(X_2,m_2)$ an order isomorphism intertwining the
associated semigroups. Denote by $h$ and $\tilde \tau$ the
associated scaling and quasi-homeomorphism. Then, the following
holds: \begin{itemize}
\item
The energy measures $\Gamma^{(c)}_1$ and $\Gamma^{(c)}_2$ are
related by
\begin{align*}
\norm{U}^2\Gamma_1^{(c)}(f)&=\tilde\tau_\#(\tilde h^2 \Gamma^{(c)}_2 (f\circ\tilde\tau))
\end{align*}
for all $f\in D(Q_1)_\loc^\bullet$.

\item The jump measures $J_1$ and $J_2$ satisfy
\begin{align*}
\norm{U}^2 J_1&=(\tilde \tau\times \tilde \tau)_\#((\tilde h\otimes \tilde h)J_2).
\end{align*}

\end{itemize}
\end{theorem}
\begin{proof}
We can assume without loss of generality that $\norm{U}=1$. By Lemma
\ref{comp_tau} there is a nest $(G_k)_{k\in\IN}$ of quasi-open
subsets of $X_1$ such that $f\circ\tau\in D(Q_2)\cap
L^\infty(X_2,m_2)$ for all $f\in \bigcup_k D(Q_1)_{G_k}\cap
L^\infty(X_1,m_1)$.

Since $h\in D(Q_2)_\loc^\bullet$, by \cite{Kuw98}, Theorem 4.1 there
exists a nest of quasi-open sets $ (G_k')$ such that  $h \in
L^\infty (G_k')$ for each $k \in \IN$. The maps $\tilde \tau, \tilde
\tau^{-1}$ are quasi-homeomorphisms that map nests to nests, see
Proposition~\ref{U_D_loc}. Therefore, they also map quasi-open sets to
quasi-open sets, see Lemma~\ref{images_quasi-homeo}. It follows from
these properties that we can assume $h\in
L^\infty(\tilde\tau^{-1}(G_k))$ for all $k\in \IN$.

Let $f,\phi\in \bigcup_k D(Q_1)_{G_k}\cap L^\infty(X_1,m_1)$. Using the intertwining property we obtain
\begin{align*}
Q_{1,\phi}(f)&=Q_1(\phi f)-Q_1(\phi f^2,\phi)\\
&=Q_2(U(\phi f))-Q_2(U(\phi f^2),U\phi)\\
&=Q_2((U\phi)(f\circ\tau))-Q_2((U\phi)(f^2\circ\tau),U\phi)\\
&=Q_{2,U\phi}(f\circ\tau).
\end{align*}

Thus,
\begin{align*}
&\int\tilde\phi^2\,d\Gamma_1^{(c)}(f)+\int\tilde\phi(x)\tilde\phi(y)\abs{\tilde f(x)-\tilde f(y)}^2\,dJ_1(x,y)\\
&=\int (\widetilde{U\phi})^2\,d\Gamma_2^{(c)}(f\circ\tau)+\int\widetilde{U\phi}(x)\widetilde{U\phi}(y)\abs{\tilde f(\tilde \tau(x))-\tilde f(\tilde \tau(y))}^2\,d J_2(x,y).
\end{align*}

By Lemma \ref{approx_D_loc} there is a sequence $(\psi_n)$ in $\bigcup_k D(Q_1)_{F_k}$ such that $\tilde\psi_n\to 1$ q.e. Define $(\phi_n)$ recursively by $\phi_1= (\psi_1\wedge 1) \vee 0,\,\phi_{n+1}=(0\vee \psi_n\wedge 1)\vee \phi_n$. Then $\phi_n\in \bigcup_k D(Q_1)_{F_k}\cap L^\infty(X_1,m_1)$ and $\tilde\phi_n\nearrow 1$ q.e.

Moreover, since  $\widetilde{U\phi_n} = \tilde h  \cdot \widetilde{\phi_n \circ \tau} = \tilde h \cdot (\tilde \varphi_n \circ \tilde \tau)$ q.e., see Lemma~\ref{h=0_polar} and Lemma~\ref{lemma:tau leaves local space invariant}, and since $\tilde \tau$ maps nests to nests, we also have $\widetilde{U\phi_n}\nearrow \tilde h$ q.e.

An application of the monotone convergence theorem and the identity $\widetilde{f \circ \tau} = \tilde f \circ \tilde \tau$, see Lemma~\ref{lemma:tau leaves local space invariant}, gives
\begin{align*}
&{}\int d\Gamma_1^{(c)}(f)+\int \abs{\tilde f(x)-\tilde f(y)}^2\,dJ_1(x,y)\\
&=\lim_{n\to\infty}\left(\int \tilde\phi_n^2\,d\Gamma_1^{(c)}(f)+\int \tilde\phi_n(x)\tilde\phi_n(y)\abs{\tilde f(x)-\tilde f(y)}^2\,dJ_1(x,y)\right)\\
&= \lim_{n\to\infty}Q_{1,\varphi_n}(f)\\
&=\lim_{n\to\infty}Q_{2,U\varphi_n}(f)\\
&=\lim_{n\to\infty}\left(\int \widetilde{U\phi_n}^2\,d\Gamma_2^{(c)}(f\circ\tau)+\int \widetilde{U\phi_n}(x)\widetilde{U\phi_n}(y)\abs{\tilde f(\tilde\tau(x))-\tilde f(\tilde \tau(y))}^2\,dJ_2(x,y)\right)\\
&=\int \tilde h^2\,d\Gamma_2^{(c)}(f\circ\tau)+\int \tilde h(x)\tilde h(y)\abs{\tilde f(\tilde\tau(x))-\tilde f(\tilde\tau(y))}^2\,dJ_2(x,y).
\end{align*}

Let
\begin{align*}
\hat \Gamma_1^{(c)}(f,g)=\tilde\tau_\#(\tilde h^2\, d\Gamma_2^{(c)}(f\circ\tau,g\circ\tau))
\end{align*}
and $\hat Q_1^{(c)}(f,g)=\hat \Gamma_1^{(c)}(f,g)(X_1)$ for $f,g\in \bigcup_k D(Q_1)_{F_k}\cap L^\infty(X_1,m_1)$. Then $\hat \Gamma_1^{(c)}$ is a positive symmetric bilinear map that satisfies
\begin{align*}
0\leq \hat \Gamma_1^{(c)}(f)(X) \leq Q_1(f)-\int \tilde h(x)\tilde h(y)\abs{\tilde f(\tilde\tau(x))-\tilde f(\tilde\tau(y))}^2\,dJ_2(x,y)\leq Q_1(f)
\end{align*}
for all $f\in \bigcup_k D(Q_1)_{F_k}\cap L^\infty(X_1,m_1)$. Hence $\hat \Gamma_1^{(c)}$ (and thus $\hat Q_1^{(c)})$ extends continuously to $D(Q_1)$ and
\begin{align*}
&\quad Q_1^{(c)}(f)+\int \abs{\tilde f(x)-\tilde f(y)}^2\,dJ_1(x,y)+\int\abs{\tilde f(x)}^2\,dk_1(x)\\
&=\hat Q_1^{(c)}(f)+\int\abs{\tilde f(x)-\tilde f(y)}^2\,d((\tilde \tau\times \tilde\tau)_\# ((h\otimes h)\,J_2))(x,y)+\int\abs{\tilde f(x)}^2\,dk_1(x)
\end{align*}
for all $f\in D(Q_1)$.

Since $\tilde\tau^{-1}$ maps polar sets to polar sets, the measure
$\hat\Gamma_1^{(c)}(f)$ does not charge polar sets for all $f\in
D(Q_1)$. Moreover, the product rule for $\Gamma_2^{(c)}$ implies
immediately the product rule for $\hat\Gamma_1^{(c)}$ on $\bigcup_k
D(Q_1)_{F_k}\cap L^\infty(X_1,m_1)$. Thus, $\hat Q_1^{(c)}$ is
strongly local by Lemma \ref{Leibniz_strongly_local}.

Therefore the uniqueness of the Beurling-Deny decomposition yields
\begin{align*}
\hat Q_1^{(c)}(f)&=Q_1^{(c)}(f)\\
J_1&=(\tilde\tau\times\tilde\tau)_\#((h\otimes h)J_2).
\end{align*}
Finally, it is not hard to see that the energy measure of $\hat Q_1^{(c)}$ is given by
\begin{align*}
\hat \Gamma_1^{(c)}(f)=\tilde\tau_\#(\tilde h^2 \Gamma_2^{(c)}(f\circ\tau))
\end{align*}
for $f\in \bigcup_k D(Q_1)_{F_k}\cap L^\infty(X_1,m_1)$. The formula for arbitrary $f\in D(Q_1)_\loc^\bullet$ follows by localization.
\end{proof}

\begin{remark}
\begin{itemize}
\item  In the special situation of graphs the transformation of the jump
parts (which are the only parts present in that situation)  is
already known (see \cite{KLSW15}, Theorem 3.6).

\item
The transformation of the killing parts can be deduced from the
formulas for the jump and strongly local part. However, it should be
noted that the killing measure of $Q_2$ may depend not only on the
killing measure of $Q_1$, but also on the jump and strongly local
part.
\end{itemize}
\end{remark}

\section{Strongly local Dirichlet forms}\label{sec_strongly_local}
In this section we consider strongly local Dirichlet forms
satisfying some additional regularity assumptions. In this situation
we show that  the transformation $\tau$ has a version that is an isometry
with respect to the intrinsic metrics (Theorem
\ref{isometry_local}). As a technical byproduct we prove an alternative formula for the intrinsic metric of regular strongly local Dirichlet forms (Proposition~\ref{intrinsic_metric_weak}), which may be of independent interest.

\medskip

A quasi-regular Dirichlet form $Q$ is called \emph{strongly local}
if the jump measure $J$ and the killing measure $k$ in the
Beurling-Deny decomposition vanish. For simplicity's sake, we then
write $\Gamma$ for the local energy measure $\Gamma^{(c)}$. In the
proofs of this section  we will freely use properties of  $\Gamma$
such as extension to local spaces and product rule discussed in
Section \ref{sec-Beurling-Deny}.

For a strongly local regular Dirichlet form $Q$, the \emph{intrinsic
metric} $d_Q$ (see \cite{BM95,Stu94}) on $X$ is defined by
\begin{align*}
d_Q(x,y)=\sup\{\abs{f(x)-f(y)}\colon f\in D(Q)_\loc\cap C(X),\,\Gamma(f)\leq m\},
\end{align*}
where $D(Q)_\loc$ is the set of all functions $f\colon X\lra\IR$
such that for every open, relatively compact $G\subset X$ there
exists $f_G\in D(Q)$ such that $f\1_G=f_G \1_G$.

Notice that despite its name containing the word \textit{metric},
$d_Q$ may attain the value $0$ off the diagonal and may be infinite
for some points.


We say that a strongly local Dirichlet form $Q$ is \emph{strictly
local} if $d_Q$ is a complete metric that induces the original
topology on $X$. We say that $Q$ satisfies the
\emph{Sobolev-to-Lipschitz property} if every $f\in D(Q)$ with
$\Gamma(f)\leq m$ has a $1$-Lipschitz version.

\begin{example}
If $Q$ is the Dirichlet energy with Dirichlet boundary conditions on
a Riemannian manifold $(M,g)$, it is folklore that $d_Q$ coincides with the
geodesic metric induced by $g$, see e.g. Proposition~2.2 of \cite{ABE12}. Thus, $Q$ is strictly local if and
only if $(M,g)$ is (geodesically) complete. Moreover, $Q$ satisfies
the Sobolev-to-Dirichlet property by the converse to Rademacher's
theorem. For an $\IR^n$ version of this property, which can be transferred to manifolds by localization, see e.g. \cite{Hei}.
\end{example}

\begin{example}
Dirichlet forms of Riemannian energy measure spaces (see
\cite{AGS15}, Def. 3.16) are strictly local and satisfy the
Sobolev-to-Lipschitz property. In particular, this applies to the
Cheeger energy on $\mathrm{RCD}(K,\infty)$ spaces in the sense of
\cite{AGS14b}. If a metric measure space satisfies the stronger $\mathrm{RCD}(K,N)$ condition for finite $N$, then it is also locally compact and the Cheeger energy is a regular Dirichlet form.
\end{example}

\begin{example}
Dirichlet forms on  metric graphs (with Kirchhoff boundary
conditions)  are strictly local and satisfy the Sobolev-to-Lipshitz
property, see e.g. \cite{LSS,Haes2}.
\end{example}

In the course of the following two lemmas we will see that the
Sobolev-to-Lipschitz property can be suitably localized.

\begin{lemma}\label{local_space_ideal}
Let $Q$ be a quasi-regular strongly local Dirichlet form on
$L^2(X,m)$ and $f\in D(Q)_\loc^\bullet\cap L^\infty(X,m)$ with $\int
d\Gamma(f)<\infty$. Then $fg\in D(Q)$ for all $g\in D(Q)\cap
L^\infty(X,m)$.
\end{lemma}
\begin{proof}
Let $(G_k)$ be a nest of quasi-open subsets of $X$ such that $f\in
D_\loc(\{G_k\})$. Let $(g_n)$ be a sequence in $\bigcup_k
D(Q)_{G_k}\cap L^\infty(X,m)$ such that
$g_n\overset{\norm{\cdot}_Q}{\to} g$ and $\norm{g_n}_\infty\leq
\norm{g}_\infty$. By definition, for all $n\in\IN$ there exists $f_n
\in D(Q) \cap L^\infty(X,m)$ with $\|f_n\|_\infty \leq \|f\|_\infty$
such that $f  = f_n$ on a quasi-open set containing  the
quasi-support of $g_n$. Then, $fg_n = f_n g_n \in D(Q)\cap
L^\infty(X,m)$ (as  $D(Q)\cap L^\infty(X,m)$ is an algebra)  and $f
g_n\to fg$ in $L^2(X,m)$.

The lower semicontinuity of $Q$, its locality  and the product-rule
for the energy measure imply
\begin{align*}
Q(fg)&\leq \liminf_{n\to\infty} Q(fg_n)\\
&=\liminf_{n\to\infty} Q(f_n g_n)\\
&\leq 2\liminf_{n\to\infty}\left(\int \tilde f_n^2\,d\Gamma(g_n)+\int \tilde g_n^2\,d\Gamma(f_n)\right)\\
&=2\liminf_{n\to\infty}\left(\int \tilde f_n^2\,d\Gamma(g_n)+\int \tilde g_n^2\,d\Gamma(f)\right)\\
&\leq 2\left(\norm{f}_\infty^2 Q(g)+\norm{g}_\infty^2 \int d\Gamma(f)\right).
\end{align*}
Hence $fg\in D(Q)$.
\end{proof}

\begin{remark}
 The space of functions that appeared in the previous lemma satisfies
 $$\{f \in  D(Q)_\loc^\bullet\cap L^\infty(X,m) \mid \int d\Gamma (f) < \infty\} = D(Q)^{\rm ref} \cap L^\infty(X,m),$$  where  $D(Q)^{\rm ref}$ is the so-called reflected Dirichlet space  of $Q$ (see \cite{Kuw02}, Definition~4.1). It is known that $D(Q)  \cap L^\infty(X,m)$ is an (multiplicative) ideal in $D(Q)^{\rm ref} \cap L^\infty(X,m)$, c.f. \cite{Kuw02}, Lemma~4.2. We included a proof for the convenience of the reader.
\end{remark}

\begin{lemma}\label{local_Sobolev-to-Lipschitz}
Let $Q$ be a strongly local regular Dirichlet form on $L^2(X,m)$
with the Sobolev-to-Dirichlet property. Assume that $d_Q$ induces
the original topology on $X$. Then, any $f\in D(Q)_\loc^\bullet\cap
L^\infty(X,m)$ with $\Gamma(f)\leq L^2 m$  has a continuous version.
\end{lemma}
\begin{proof}
Let $(G_k)$ be an ascending sequence of open, relatively compact
subsets of $X$ with $\bigcup_k G_k=X$. Let $H_k=\{x\in G_k\mid
d(x,G_k^c)>\frac 1 k\}$ and $\phi_k=(1-k d_Q(x,H_k))_+$. By Lemma
1.9 of \cite{Stu95} (see Lemma 1 of \cite{Stu94} as well), the function
$\phi_k$ then belongs to $D(Q)_\loc$ and satisfies
$\Gamma(\phi_k)\leq k^2 m$. Moreover, we clearly have
$\supp\phi_k\subset \overline{G_k}$. Hence $\phi_k$ belongs to
$D(Q)\cap L^\infty_c(X,m)$ as any function from $D(Q)_\loc$ with
compact support belongs to $D(Q)$ by Theorem 3.5 of \cite{FLW14}.

By the product rule and Cauchy-Schwarz inequality we have
$f\phi_k\in D(Q)_\loc^\bullet\cap L^\infty_c(X,m)$ and
\begin{align*}
\Gamma(f\phi_k)=\tilde f^2\Gamma(\phi_k)+\tilde\phi_k^2\Gamma(f)+2\tilde f\tilde \phi_k\Gamma(f,\phi_k)\leq 2(k^2\norm{f}_\infty^2+L^2)m.
\end{align*}
Moreover, the locality of $\Gamma$ implies $\1_{G_{k+1}^c} \Gamma(f\varphi_k)  = 0$  and so $\int d\Gamma(f\phi_k)<\infty$. Since $G_k$ is relatively compact and $Q$ regular, we can pick $\psi_k\in D(Q)\cap L^\infty(X,m)$ such that $\psi_k \1_{G_k}=\1_{G_k}$.
Lemma \ref{local_space_ideal} implies $f\phi_k=f\phi_k\psi_k\in D(Q)$.

By the Sobolev-to-Lipschitz property there is a Lipschitz version
$\tilde f_k$ of $f\phi_k$. Since $\tilde f_{k+1}\1_{G_k}=\tilde f_k
\1_{G_k}$ pointwise, the limit $\tilde f(x)=\lim_{k\to\infty} \tilde
f_k(x)$ exists, is continuous and coincides with $f$ a.e. on
$\bigcup_k G_k=X$.
\end{proof}

The following lemma may well be of independent interest.  It shows that in the definition of $d_Q$ the space $D(Q)_\loc$ can be replaced by the larger $D(Q)^\bullet_\loc$ if $Q$ is regular.

\begin{lemma}\label{intrinsic_metric_weak}
Let $Q$ be a strongly local regular Dirichlet form on $L^2(X,m)$. If $f\in C(X)\cap D(Q)_\loc^\bullet$ with $\Gamma(f)\leq m$, then $f\in D(Q)_\loc$. In particular, the intrinsic metric $d_Q$ can be expressed as
\begin{align*}
d_Q(x,y)=\sup\{\abs{f(x)-f(y)}\colon f\in C_b(X)\cap D(Q)_\loc^\bullet,\,\Gamma(f)\leq m\}.
\end{align*}
\end{lemma}
\begin{proof}
Let $G\subset X$ be relatively compact and open. Since $Q$ is
regular, there exists $\phi\in C_c(X)\cap D(Q)$ such that $\1_G\leq
\phi\leq 1$. By the product rule and Cauchy-Schwarz inequality, we
have $\Gamma(f\phi)\leq 2f^2\Gamma(\phi)+2\phi^2\Gamma(f)$. Since
$f$ is locally bounded, the support of $\Gamma(\phi)$ is compact and
$\Gamma(f)\leq m$, it follows that $\int d\Gamma(f\phi)<\infty$. Now
we infer from Lemma \ref{local_space_ideal} that $f\phi^2\in D(Q)$,
and clearly, $f\phi^2|_G=f|_G$.

That one can restrict the optimization problem for $d_Q$ to bounded functions follows by a standard cut-off argument.
\end{proof}

\begin{proposition}\label{tau_strongly_local}
Let $Q_i$, $i\in\{1,2\}$, be regular irreducible local Dirichlet
forms on $L^2(X_i,m_i)$ with the Sobolev-to-Lipschitz property, and
assume that each  $d_{Q_i}$ induces the original topology on $X_i$.
Let $U\colon L^2(X_1,m_1)\lra L^2(X_2,m_2)$ be an order isomorphism
intertwining the associated semigroups with associated
transformation $\tau$.

Then there exist closed polar sets $N_i\subset X_i$, $i\in\{1,2\}$,
and a version $\tilde\tau$ of $\tau$ such that
\begin{align*}
\tilde\tau\colon X_2\setminus N_2\lra X_1\setminus N_1
\end{align*}
is a homeomorphism.
\end{proposition}
\begin{proof}
The proof is similar to that of Theorem \ref{tau_quasi-homeo}.
Denote by $\mathcal{L}$ the Lipschitz functions on $X_1$ with
compact support. By Theorem 4.9 of \cite{FLW14} we have
$\mathcal{L}\subset D(Q_1)_\loc$ and $\Gamma_1(f)\leq L^2 m_1$
whenever  $f\in\mathcal{L}$ is $L$-Lipschitz. As functions from
$D(Q_1)_\loc$ with compact support belong to $D(Q)$ by  Theorem 3.5
of \cite{FLW14},  we then also find $\mathcal{L}\subset D(Q_1)$.

By Corollary~\ref{transformation_measures} and
Theorem~\ref{transformation_decomposition} we have $f\circ\tau\in
D(Q_2)_\loc^\bullet$ and $\Gamma_2(f\circ\tau)\leq L m_2$ for all
$f\in \mathcal{L}$.  By Lemma \ref{local_Sobolev-to-Lipschitz},
$f\circ\tau$ has a continuous version $\widetilde{f\circ\tau}$ for
all $f\in \mathcal{L}$.

Let $N_2=\{y\in Y\mid \widetilde{f\circ\tau}(y)=0\text{ for all
}f\in\mathcal{L}\}$. Then $N_2$ is closed as the intersection of
closed sets. It is polar by the same arguments as in Lemma
\ref{h=0_polar}.

By the Stone-Weierstraß theorem, $\mathcal{L}$ is dense in
$C_0(X_1)$. Just as in the proof of Theorem \ref{tau_quasi-homeo} we
get a map $\tilde\tau\colon X_2\setminus N_2\lra X_1$ such that
$\widetilde{f\circ\tau}=f\circ\tilde\tau$ for all $f\in
\mathcal{L}$. Continuity follows similarly since for every $x\in
X_1$ the bump function $(1-\frac 1 rd(x,\cdot))_+$ is in
$\mathcal{L}$ for all $r>0$. An application of the same arguments to
$U^{-1}$ with its transformation $\tau^{-1}$ yields the claim.
\end{proof}

\begin{theorem}[Isometric transformation]\label{isometry_local}
Let $Q_i$, $i\in\{1,2\}$, be strongly local regular irreducible
Dirichlet forms on $L^2(X_i,m_i)$ with the Sobolev-to-Lipschitz
property. Assume that each $d_{Q_i}$ induces the original topology
on $X_i$ and denote by $\overline{X_i}$ the completion of $X_i$
w.r.t. $d_{Q_i}$.

If $U\colon L^2(X_1,m_1)\lra L^2(X_2,m_2)$ is an order
isomorphism intertwining $Q_1$ and $Q_2$, then the
associated transformation $\tau$ has a version $\tilde\tau$ that extends to an
isometry from $\overline{X_2}$ onto $\overline{X_1}$. In particular, if $Q_1$ and $Q_2$ are strictly local, then $(X_1,d_{Q_1})$ and $(X_2,d_{Q_2})$ are isometric.
\end{theorem}
\begin{proof}
For $f\in C_b(X_1)\cap D(Q_1)_\loc^\bullet$ with $\Gamma_1(f)\leq m_1$ we have
$f\circ\tau\in D(Q_2)_\loc^\bullet\cap L^\infty(X_2,m_2)$ by Lemma \ref{lemma:tau leaves local space invariant} and $\Gamma_2(f\circ\tau)\leq
m_2$ by Corollary \ref{transformation_measures} and Theorem \ref{transformation_decomposition}. Now, chose $\tilde \tau$ according to Proposition
\ref{tau_strongly_local}. Then,  $f\circ\tilde\tau$ is continuous on
$X_2\setminus N_2$ by Proposition \ref{tau_strongly_local}.
Furthermore, $f\circ\tilde\tau$ has a continuous extension to $X_2$
by Lemma \ref{local_Sobolev-to-Lipschitz}.

Using Lemma \ref{intrinsic_metric_weak}, we see that for all $x,y\in X_2\setminus N_2$ we have
\begin{align*}
d_{Q_1}(\tilde\tau(x),\tilde\tau(y))&=\sup\{ \abs{f(\tilde\tau(x))-f(\tilde\tau(y))}\colon f\in C_b(X_1)\cap D(Q_1)_\loc^\bullet,\,\Gamma_1(f)\leq m_1\}\\
&\leq \sup\{\abs{g(x)-g(y)}\colon g\in C_b(X_2)\cap D(Q_2)_\loc^\bullet,\,\Gamma_2(g)\leq m_2\}\\
&=d_{Q_2}(x,y).
\end{align*}
The converse inequality follows by exchanging the roles of $Q_1$ and $Q_2$.

The sets $N_i$ from Proposition \ref{tau_strongly_local} are in particular null sets, hence $X_i\setminus N_i$ is dense in $\overline{X_i}$ since $m_i$ has full support. Thus $\tilde\tau$ extends to an isometry from $\overline{X_2}$ onto $\overline{X_1}$.
\end{proof}

\begin{remark}
\begin{itemize}
 \item The previous theorem can be seen as an  extension of the main result
of \cite{ABE12} (which deals with manifolds)  to strongly local regular
Dirichlet forms. In fact, even  in the smooth setting, our result
breaks  new ground as it covers \emph{weighted} Riemannian
manifolds. Let us note, however, that in order to obtain an isometry $X_2 \to X_1$ (instead of their completions), our assumptions are slightly
more restrictive than those from \cite{ABE12}. We assume completeness of the space, whereas in the mentioned work only some form of regularity of the
boundary of the metric completion is assumed. Our result should also hold true in this more general setting  with the regularity condition of the metric boundary adapted to  Dirichlet forms.
\item It is clear that some form of regularity of the space is needed in oder to obtain an isometry on the underlying spaces and not only on completions.  Indeed, for a given regular Dirichlet form one can remove a polar set from the underlying space and consider its trace on the remaining part. By the polarity of the set the generated semigroups coincide  but the underlying spaces do not.
\end{itemize}

\end{remark}

\section{Non-local Dirichlet forms and recurrence}\label{sec_non-local}
In this section we study intertwining of general (i.e. not
necessarily local) regular Dirichlet forms. In this context metric
considerations are also possible due to the  recently developed
theory of intrinsic metric for general regular Dirichlet forms
\cite{FLW14}. The considerations of \cite{FLW14} do not  provide one
intrinsic metric but rather a whole family of  intrinsic metrics.
The strongly local case is then distinguished by there being one
largest intrinsic metric (referred to as `the' intrinsic metric in
the previous section). For general regular Dirichlet forms the
intrinsic metrics will in general not be  compatible and we will
have to deal with the whole family. As main result in this section
we establish a bijective correspondence between the families of
intrinsic metrics for intertwined Dirichlet forms (Theorem
\ref{isometry_non-local}) under suitable conditions.

\medskip

We say that a subset $F$ of $L^0(X,m)$ is a \emph{core of bump
functions} if every $f\in F$ has a continuous representative and
$F\cap C_0(X)$ is dense in $C_0(X)$.

\begin{lemma}
Let $Q_i$ be recurrent, irreducible, quasi-regular Dirichlet forms
on $L^2(X_i,m_i)$, $i\in\{1,2\}$, and let $U\colon L^2(X_1,m_1)\lra
L^2(X_2,m_2)$ be an order isomorphism intertwining the associated
semigroups. Assume that there exist cores of bump functions
$F_i\subset L^2(X_i,m_i)$, $i\in\{1,2\}$, such that $UF_1=F_2$.
Then, there  are closed, polar sets $N_i\subset X_i$, $i\in\{1,2\}$,
and a version $\tilde\tau$ of $\tau$ such that
\begin{align*}
\tilde\tau\colon X_2\setminus N_2\lra X_1\setminus N_1
\end{align*}
is a homeomorphism.
\end{lemma}
\begin{proof}
By Corollary \ref{recurrent_h_const}, the scaling $h$ is constant. Since $UF_1=F_2\subset C(X_2)$, we have $f\circ\tau\in C(X_2)$ for all $f\in F_1\cap C_0(X_1)$. From here on we can proceed exactly as in Proposition \ref{tau_strongly_local}.
\end{proof}

The condition on the existence of suitable cores of bump functions in the proposition above does not only depend on the Dirichlet forms, but also on the order isomorphism. However, for several classes of Dirichlet forms, this condition is satisfied for all intertwining order isomorphism.
\begin{example}
If $(X_i,b_i,c_i,m_i)$, $i\in\{1,2\}$, are connected weighted graphs
and $Q_i$ the associated Dirichlet forms with Dirichlet boundary
conditions (see \cite{KL12}), then $F_i:=D(Q_i)$ is a core of bump
functions since $C_c(X_i)\subset D(Q_i)\subset C(X_i)$. By
definition, any order isomorphism intertwining $Q_1$ and $Q_2$ maps
$F_1$ to $F_2$.
\end{example}
\begin{example}
Let $(M_i,g_i)$ be Riemannian manifolds, $\alpha\in(0,1]$, and $Q_i$ the Dirichlet form generated by the fractional Laplacian $(-\Delta_i)^{\alpha/2}$ on $M_i$. The space $F_i=\bigcap_{k=1}^\infty D((-\Delta_i)^{n\alpha/2})$ is a core of bump functions by the Sobolev embedding theorem. Moreover, any order isomorphism intertwining $Q_1$ and $Q_2$ maps $F_1$ to $F_2$.
\end{example}

Let $Q$ be a regular Dirichlet form on $L^2(X,m)$. For an open subset $V$ of $X$ define the Dirichlet form $Q^{(V)}$ on $L^2(V,m)$ as the closure of the restriction of $Q$ to $D(Q)\cap C_c(V)$. If $X\setminus V$ is polar, then
\begin{align*}
Q(f)=\int_V d\Gamma^{(c)}(f)+\int_{(V\times V)\setminus \Delta_V}(f(x)-f(y))^2\,dJ(x,y)+\int_V f^2\,dk
\end{align*}
for all $f\in D(Q^{(V)})\cap C_c(V)$. From the uniqueness of the Beurling-Deny decomposition we can infer that the strongly local energy measure, jump measure and killing measure for $Q^{(V)}$ are given by the restrictions the measures for $Q$ to $V$ resp. $(V\times V)\setminus \Delta_V$.

We now turn to intrinsic metrics in this context  following the
theory developed in \cite{FLW14}. As mentioned already, when dealing
with non-local Dirichlet forms, there is not a distinguished
intrinsic metric, but a family of intrinsic metrics. Let us briefly
recall the relevant definitions.

Recall that the local form domain $D(Q)_\loc$ of a regular Dirichlet form is the set of all functions $f\colon X\lra\IR$ such that for every open, relatively compact $G\subset X$ there exists $f_G\in D(Q)$ such that $f\1_G=f_G \1_G$.

Let $J$ be the
jump measure of $Q$ in the Beurling-Deny decomposition. The space
$D(Q)_\loc^\ast$ is the set of all functions $f\in D(Q)_\loc$ such
that
\begin{align*}
\int_{(K\times X)\setminus \Delta_X}(\tilde f(x)-\tilde f(y))^2\,dJ(x,y)<\infty
\end{align*}
for all compact $K\subset X$.

For $f\in D(Q)_\loc^\ast$ the Radon measure $\Gamma^{(b)}(f)$ is defined by
\begin{align*}
\Gamma^{(b)}(f)(E)=\int_{(E\times X)\setminus \Delta_X}(\tilde u(x)-\tilde u(y))^2\,dJ(x,y)
\end{align*}
for $E\subset X$ measurable.

A pseudo-metric $d\colon X\times X\lra[0,\infty]$ is called \emph{intrinsic metric} for $Q$ if there are Radon measure $m^{(b)}$ and $m^{(c)}$ with $m^{(b)}+m^{(c)}\leq m$ such that for all $A\subset X$ and all $M>0$ the function $d_A=d(\,\cdot\,,A)$ satisfies
\begin{itemize}
\item $d_A\wedge M\in D(Q)_\loc^\ast\cap C(X)$,
\item $\Gamma^{(b)}(d_A\wedge T)\leq m^{(b)}$,
\item $\Gamma^{(c)}(d_A\wedge T)\leq m^{(c)}$.
\end{itemize}
The set of all intrinsic metrics for $Q$ is denoted by $\mathcal I(Q)$.

If $Q$ is strongly local, then the metric $d_Q$ discussed in Section \ref{sec_strongly_local} is an intrinsic metric in the sense of the definition above, and every intrinsic metric for $Q$ is pointwise dominated by $d_Q$ (see \cite{FLW14}, Theorem 6.1).

\begin{theorem}\label{isometry_non-local}
Let $Q_i$ be recurrent, irreducible, regular Dirichlet forms on
$L^2(X_i,m_i)$, $i \in \{1,2\}$. Let $U\colon L^2(X_1,m_1)\lra
L^2(X_2,m_2)$ be an order isomorphism intertwining the associated
semigroups and denote by $\tau$ the associated transformation.
Assume that there exist cores of bump functions $F_i\subset
L^2(X_i,m_i)$ such that $U F_1=F_2$. Let $N_i\subset X_i$ be closed
polar sets, $V_i=X_i\setminus N_i$, and $\tilde\tau$ a version of
$\tau$ such that $\tilde\tau\colon V_2\lra V_1$ is a homeomorphism.
Then
\begin{align*}
\Phi\colon\mathcal I(Q_1^{(V_1)})\lra \mathcal I(Q_2^{(V_2)}),\,d\mapsto d(\tilde\tau(\cdot),\tilde\tau(\cdot))
\end{align*}
is a bijection.
\end{theorem}
\begin{proof}
It suffices to show that $\Phi$ maps intrinsic metrics to intrinsic
metrics. Let $d_1$ be an intrinsic metric for $Q^{(V_1)}_1$ and let
$m_1^{(b)}, m_1^{(c)}$ be Radon measures on $X_1$ such that
$m_1^{(b)}+m_1^{(c)}\leq m_1$ and $\Gamma_1^{(b)}((d_1)_A\wedge
M)\leq m_1^{(b)}$, $\Gamma_1^{(c)}((d_1)_A\wedge M)\leq m_1^{(c)}$
for all $A\subset U_1$ and $M>0$.

Let $d_2=\Phi(d_1)$. Since $Q_1$, $Q_2$ are recurrent, there is an $\alpha>0$ with $h=\alpha$ a.e. It follows easily from Theorem \ref{transformation_decomposition} that $(d_2)_{\tilde\tau^{-1}(A)}\wedge M=((d_1)_A\wedge M)\circ\tilde \tau\in D(Q_2)_\loc^\ast$ for all $A\subset U_1$ and $M>0$. Moreover, the fact that $\tilde\tau\colon U_2\lra U_1$ is a homeomorphism guarantees $(d_2)_A\wedge T\in C(X_2)$.

Let $m_2^{(b)}=\alpha^2 m_1^{(b)}$, $m_2^{(c)}=m_1^{(c)}$. An application of Theorem \ref{transformation_decomposition} gives $m_2^{(b)}+m_2^{(c)}\leq m_2$ and $\Gamma_2^{(b)}((d_1)_{\tilde\tau^{-1}(A)}\wedge M)\leq m_2^{(b)}$, $\Gamma_2^{(c)}((d_1)_{\tilde\tau^{-1}(A)}\wedge M)\leq m_2^{(c)}$ for all $A\subset V_1$ and $M>0$. Hence $d_2$ is an intrinsic metric.
\end{proof}

\begin{remark}
When points have positive capacity, the sets $V_1$, $V_2$ coincide
with $X_1$, $X_2$. Consequently, in this case  $\tilde\tau$ induces
a bijection between the sets of intrinsic metrics for the original
Dirichlet forms. This result is new even for the case of  graphs
treated in \cite{KLSW15}.
\end{remark}

\section{Resistance forms}\label{sec-Resistance}
In this section we study Dirichlet forms induced by resistance
forms. The theory of such forms goes back to Kigami  and plays  a
major role in analysis and probability theory on fractals
 \cite{Kig01,Kig12}.
Resistance forms naturally come with a metric viz. the resistance
metric and this metric is a fundamental tool in their study.  As
main result of this section we show that the transformation
associated with an intertwining order isomorphism is an isometry (up
to an overall factor) with respect to the resistance metrics
(Theorem \ref{isometry_resistance}).

\medskip

A \emph{resistance form} on a set $X$ is a pair $(\E,\F)$ consisting
of a subspace $\F\subset \IR^X$ and a positive quadratic form $\E$
on $\F$ such that the following conditions hold:
\begin{itemize}
\item The constant functions are contained in $\F$ and $\E(f)=0$ if and only if $f$ is constant.
\item The quotient $\F/\IR1$ is a Hilbert space with norm $\E^{1/2}$.
\item If $V\subset X$ is finite and $g\colon V\lra\IR$, there is a function $f\in \F$ such that $f|_V=g$.
\item For all $x,y\in X$, the distance
\begin{align*}
R(x,y)=\sup\{\abs{f(x)-f(y)}^2\mid f\in\F,\,\E(f)\leq 1\}
\end{align*}
is finite.
\item If $f\in\F$, then $\bar f:=(f\vee 0)\wedge 1\in\F$ and $\E(\bar f)\leq \E(f)$.
\end{itemize}

It is easily seen that $R$ is a metric on $X$, the so-called
resistance metric associated with $(\E,\F)$. We will endow $X$ with
the topology induced by $R$.

The resistance form $(\E,\F)$ is called regular if $X$ is locally
compact and $\F\cap C_c(X)$ is dense in $C_0(X)$. By \cite{Kig12},
Theorem 6.3, this condition is equivalent to the following:

If $K\subset X$ is compact and $U\subset X$ open such that $K\subset
U$ and $\bar U$ is compact, then there is a function $\phi\in\F$
such that $\phi|_K=1$, $\operatorname{im}\phi\subset[0,1]$ and
$\operatorname{supp}\phi\subset \bar U$.

Given a Radon measure $m$ of full support on $(X,R)$, the
restriction of $\E$ to $\F\cap L^2(X,m)$ is closed in $L^2(X,m)$
(cf. \cite{Kig01}, 2.4.1).

We call $Q$ a Dirichlet form associated with $(\E,\F)$ if $C_c(X)\cap\F\subset D(Q)\subset \F\cap L^2(X,m)$ and $Q$ is a restriction of $\E$. If $(\E,\F)$ is regular, the closure of the restriction of $\F$ to $C_c(X)\cap\F$ is a regular Dirichlet form.

\begin{lemma}\label{resistance_recurrent}
Let $(\E,\F)$ be a resistance form on $X$ and $m$ a finite Radon
measure on $X$. Then $\E$ restricted to $\F\cap L^2(X,m)$ is an
irreducible, recurrent Dirichlet form.
\end{lemma}
\begin{proof}
Denote by $Q$ the restriction of $\E$ to $\F\cap L^2(X,m)$. We have $1\in \F\cap L^2(X,m)=D(Q)\subset D(Q)_e$ and $Q(1)=\E(1)=0$. Thus, $Q$ is recurrent.

If $A\subset X$ is a $Q$-invariant set, then $\1_A\in D(Q)\subset \F$ and $Q(\1_A)=Q(1)-Q(\1_{A^c})\leq 0$, hence $\1_A=0$ or $\1_A=1$. Thus, $Q$ is irreducible.
\end{proof}

\begin{lemma}
Let $(\E,\F)$ be a regular resistance form. Then the resistance metric is given by
\begin{align*}
R(x,y)=\sup\{\abs{f(x)-f(y)}^2\colon f\in \F\cap C_b(X),\,\E(f)\leq 1\},\;x,y\in X.
\end{align*}
\end{lemma}
\begin{proof}
Let $x,y\in X$. For $\epsilon>0$ choose $g\in\F$ such that $\E(g)\leq 1$ and $\abs{g(x)-g(y)}^2\geq R(x,y)-\epsilon$. Assume without loss of generality that $g(x)\geq g(y)$. Let $f=(g\wedge g(x))\vee g(y)$. Then $f\in \F\cap C_b(X)$, $\E(f)\leq\E(g)\leq 1$ and
\begin{align*}
\abs{f(x)-f(y)}^2=\abs{g(x)-g(y)}^2\geq R(x,y)-\epsilon.
\end{align*}
Since $\epsilon$ was arbitrary, the assertion follows.
\end{proof}

Here comes the main result of this section. It essentially shows
that intertwining order isomorphisms between resistance forms are
induced by  isometries w.r.t. resistance metric  between the spaces.

\begin{theorem}[$U$ is induced from an isometry] \label{isometry_resistance}
Let $(\E_i,\F_i)$ be regular resistance forms on $X_i$,
$i\in\{1,2\}$, $m_i$ finite Borel measures of full support on
$(X_i,R_i)$ and $Q_i$ irreducible, recurrent Dirichlet forms
associated with $(\E_i,\F_i)$. If $U\colon L^2(X_1,m_1)\lra
L^2(X_2,m_2)$ is an order isomorphism intertwining $Q_1$ and $Q_2$,
then there is a version $\tilde\tau$ of the associated
transformation $\tau$, which is a homeomorphism,  and a constant
$\alpha>0$ such that $h=\alpha$ a.e. and
\begin{align*}
\norm{U}^2 R_1(\tilde \tau(y), \tilde \tau(z))=\alpha^2 R_2(y,z)
\end{align*}
for all $y,z\in X_2$.
\end{theorem}
\begin{proof}

Corollary \ref{recurrent_h_const} implies that there is a constant
$\alpha>0$ such that $h=\alpha$ a.e.

Since $m_i$ are finite, we have $\F_i\cap C_b(X_i)=D(Q_i)\cap
C_b(X_i)$. Moreover, by \cite{Kig12}, Proposition 9.13,
$Q_i$-quasi-continuous functions are continuous. Hence $\tau$ has a
version $\tilde\tau$ that is a homeomorphism. Furthermore,
$U(D(Q_1)\cap C_b(X_1))=D(Q_2)\cap C_b(X_2)$.

If $f\in D(Q_1)$, it follows that
\begin{align*}
Q_2(\alpha f\circ\tau)=Q_2(Uf)=\norm{U}^2 Q_1(f).
\end{align*}
Thus, $Q_1(f)\leq 1$ if and only if $\alpha ^2Q_2(f\circ\tau)\leq
\norm{U}^2$. Therefore,
\begin{align*}
R_1(\tilde\tau(y),\tilde\tau(z))&=\sup\{\abs{f(\tilde\tau(y))-f(\tilde\tau(z))}^2\mid f\in D(Q_1)\cap C_b(X_1),\,Q_1(f)\leq 1\}\\
&=\sup\{\abs{g(y)-g(z)}^2\mid g\in D(Q_2)\cap C_b(X_2),\,\alpha^2 Q_2( g)\leq \norm{U}^2\}\\
&=\frac{\alpha^2}{\norm{U}^2}R_2(y,z).\qedhere
\end{align*}
\end{proof}
  As a consequence of the previous
considerations we can treat typical examples of resistance forms as
arising in the study of fractals. This is the content of the next
corollary.

\begin{corollary}[Typical example]
Let $(\E_i,\F_i)$ be resistance forms on $X_i$, $i\in\{1,2\}$, such
that $(X_i,R_i)$ are compact, $m_i$ probability measures on $X_i$
and $Q_i$ Dirichlet forms associated with $\E_i$. If $U\colon
L^2(X_1,m_1)\lra L^2(X_2,m_2)$ is an order isomorphism intertwining
$Q_1$ and $Q_2$, then $\tilde\tau$ is an surjective  isometry with
respect to the resistance metrics $R_1$, $R_2$.
\end{corollary}
\begin{proof}
By Lemma \ref{resistance_recurrent}, the Dirichlet forms $Q_1$ and
$Q_2$ are irreducible and recurrent. It suffices to show that
$\alpha=\norm{U}$ in Theorem \ref{isometry_resistance}. By Corollary
\ref{transformation_measures} we have $\alpha^2 \tau_\#
m_2=\norm{U}^2 m_1$. Since $m_1(X_1)=1 = m_2(X_2)$, the assertion
follows.
\end{proof}

\begin{remark} Inspection of the proof of the corollary shows that
the compactness assumption on the $X_i$ is not necessary. Also, it
is not necessary that the $m_i$ are probability measures. It
suffices that they are finite measures with $m_1 (X_1) = m_2 (X_2)$.
\end{remark}

\begin{remark} The set of resistance forms is not disjoint from  the
sets of strictly local forms. So, for strictly local resistance
forms one could also try and use the intrinsic metrics discussed
above. However, in typical situations these intrinsic metrics will
be zero and, hence, do  not capture any geometry of the underlying
set \cite{Hin05}.
\end{remark}


\end{document}